\documentclass{amsart}

\usepackage{amsmath,amssymb,amsfonts,amssymb,amsthm}

\usepackage{verbatim}
\usepackage[usenames]{color}
\usepackage{hyperref}
\usepackage{url}
\usepackage{tikz,tikz-qtree,ifthen,cancel}
\usepackage{array,tikz-qtree,ifthen,cancel}
\usepackage{graphicx}
\usepackage{adjustbox}
\usepackage{amsthm,graphicx,tikz,appendix,tikz-qtree,ifthen,cancel}
\usetikzlibrary{calc,shapes,patterns,positioning}

\usepackage{yfonts}

\DeclareFontFamily{U}{mathb}{\hyphenchar\font45}
\DeclareFontShape{U}{mathb}{m}{n}{
      <5> <6> <7> <8> <9> <10> gen * mathb
      <10.95> mathb10 <12> <14.4> <17.28> <20.74> <24.88> mathb12
}{}
\DeclareSymbolFont{mathb}{U}{mathb}{m}{n}
\DeclareMathSymbol{\llcurly}{3}{mathb}{"CE}
\DeclareMathSymbol{\ggcurly}{3}{mathb}{"CF}
\DeclareFontFamily{U}{matha}{\hyphenchar\font45}
\DeclareFontShape{U}{matha}{m}{n}{
      <5> <6> <7> <8> <9> <10> gen * matha
      <10.95> matha10 <12> <14.4> <17.28> <20.74> <24.88> matha12
      }{}
\DeclareSymbolFont{matha}{U}{matha}{m}{n}
\DeclareMathSymbol{\curlywedge} {2}{matha}{"4E}
\DeclareMathSymbol{\curlyvee} {2}{matha}{"4F}

\newtheorem{thm}{Theorem}[section]

\newtheorem*{thmmain}{Theorem \ref{thm.main}}
\newtheorem*{thmEll}{Theorem \ref{thm.SDAPEllentuck}}
\newtheorem{lem}[thm]{Lemma}
\newtheorem{cor}[thm]{Corollary}

\newtheorem*{thm7.3}{Theorem \ref{thm.MillikenSWP}}

\newtheorem{fact}[thm]{Fact}

\theoremstyle{remark}
\newtheorem{rem}[thm]{Remark}

\theoremstyle{definition}
\newtheorem{defn}[thm]{Definition}
\newtheorem{conv}[thm]{Convention}
\newtheorem{convention}[thm]{Convention}
\newtheorem{notation}[thm]{Notation}

\newtheorem{assumption}[thm]{Assumption}

\theoremstyle{remark}

\newcommand{\EEAP}{SDAP}
\newcommand{\SFAP}{SFAP}
\newcommand{\SFAPfull}{Substructure Free Amalgamation Property}
\newcommand{\SDAP}{\rm SDAP}
\newcommand{\SDAPfull}{Substructure Disjoint Amalgamation Property}
\newcommand{\SDAPplus}{SDAP$^+$}

\newcommand{\al}{\alpha}
\newcommand{\om}{\omega}

\newcommand{\sse}{\subseteq}
\newcommand{\contains}{\supseteq}
\newcommand{\forces}{\Vdash}

\newcommand{\AKD}{\mathcal{A}\bK(\bfD)}

\DeclareMathOperator{\tree}{tree}

\DeclareMathOperator{\splitpred}{sp}

\DeclareMathOperator{\ran}{ran}
\DeclareMathOperator{\dom}{dom}

\DeclareMathOperator{\depth}{depth}

\DeclareMathOperator{\Sim}{Sim}
\DeclareMathOperator{\Ext}{Ext}

\DeclareMathOperator{\type}{tp}

\newcommand{\LSDAP}{\rm LSDAP}

\newcommand{\re}{\restriction}

\newcommand{\bD}{\mathbb{D}}
\newcommand{\bE}{\mathbb{E}}
\newcommand{\bP}{\mathbb{P}}

\newcommand{\bQ}{\mathbb{Q}}
\newcommand{\bT}{\mathbb{T}}
\newcommand{\bS}{\mathbb{S}}

\newcommand{\bU}{\mathbb{U}}

\newcommand{\bA}{\mathbf{A}}

\newcommand{\bJ}{\mathbf{J}}

\newcommand{\bK}{\mathbf{K}}
\newcommand{\bM}{\mathbf{M}}
\newcommand{\bN}{\mathbf{N}}

\newcommand{\K}{\mathrm{K}}
\newcommand{\A}{\mathrm{A}}
\newcommand{\B}{\mathrm{B}}
\newcommand{\C}{\mathrm{C}}

\newcommand{\bfA}{\mathbf{A}}
\newcommand{\bfB}{\mathbf{B}}
\newcommand{\bfC}{\mathbf{C}}
\newcommand{\bfD}{\mathbf{D}}
\newcommand{\bfE}{\mathbf{E}}

\newcommand{\bfM}{\mathbf{M}}

\newcommand{\plussim}{\stackrel{+}{\sim}}

\newcommand{\ra}{\rightarrow}

\newcommand{\lgl}{\langle}
\newcommand{\rgl}{\rangle}

\newcommand{\rl}{\upharpoonleft}

\newcommand{\Erdos}{Erd{\H{o}}s}
\newcommand{\Fraisse}{Fra{\"{i}}ss{\'{e}}}

\newcommand{\noprint}[1]{\relax}

\title[Ramsey theory of
homogeneous  structures]{Infinite-dimensional Ramsey theory for homogeneous structures with SDAP$^+$}

\author{Natasha Dobrinen}
\address{Department of Mathematics\\
 University of Notre Dame \\
255 Hurley Bldg\\
Notre Dame, IN 46556   U.S.A.}
\email{ndobrine@nd.edu}
\thanks{This research was supported by  National Science Foundation Grant DMS-1901753}

\subjclass[2020]{05D10, 05C55,  05C15, 05C05,  03C15, 03E75}

\begin{document}

\maketitle
\tableofcontents

\begin{abstract}
We prove  that  
for any    homogeneous structure $\bK$
 in a language with finitely many relation symbols of arity at most two  satisfying \SDAPplus\ (or LSDAP$^+$),
there are spaces of subcopies of $\bK$, forming subspaces of the Baire space,
in which all Borel sets are Ramsey.
Structures satisfying SDAP$^+$ include the rationals, the Rado graph and more generally, unrestricted structures, and generic $k$-partite graphs,  the latter three types with or without an additional dense linear order.
As a corollary of the main theorem,
we obtain an analogue of the Nash-Williams Theorem which recovers
exact  big Ramsey degrees for these structures,
 answering a question raised by Todorcevic 
at  the 2019 Luminy Workshop on Set Theory.
Moreover,   for the rationals and 
 similar homogeneous structures  
our methods produce topological Ramsey spaces, thus
 satisfying
analogues of the Ellentuck theorem.  
\end{abstract}


\section{Introduction}\label{section.intro}

Ramsey theory
was initiated  by
 the following celebrated result.

\begin{thm}[Ramsey \cite{Ramsey30}]\label{thm.RamseyInfinite}
Given  a positive integer $k$,
suppose
that $[\mathbb{N}]^k$,
 the collection of  all $k$-element subsets of the natural numbers, 
is partitioned into finitely many  pieces.
Then there is an infinite  subset  $N\sse\mathbb{N}$ such that 
$[N]^k$
is  contained in one  piece of the partition.
\end{thm}

Extensions of Ramsey's Theorem to colorings of 
infinite subsets of $\mathbb{N}$ have been proved, subject to constraints necessitated by the Axiom of Choice.
Considering the set of all infinite subsets of the natural numbers, denoted by $[\mathbb{N}]^{\mathbb{N}}$, as the Baire space with its metric topology,
a set $\mathcal{X}\sse [\mathbb{N}]^{\mathbb{N}}$ is called {\em Ramsey} if for each $M\in  [\mathbb{N}]^{\mathbb{N}}$, there is an $N\in [M]^{\mathbb{N}}$ such that either $[N]^{\mathbb{N}}\sse\mathcal{X}$ or else 
$[N]^{\mathbb{N}}\cap\mathcal{X}=\emptyset$.
From 1965 through 1974, a 
beautiful progression of results  was obtained
using topological properties  to guarantee that certain subsets of the Baire space are Ramsey. 
Nash-Williams  proved that clopen sets are Ramsey
in \cite{NashWilliams65};
Galvin and Prikry proved that Borel sets are Ramsey in \cite{Galvin/Prikry73}; and 
Silver extended this to analytic sets in \cite{Silver71}.
This line of work culminated in the topological characterization of Ramsey sets found by Ellentuck in \cite{Ellentuck74}  in terms of a topology refining the metric topology on $[\mathbb{N}]^{\mathbb{N}}$, now referred to as the Ellentuck topology.

This paper is focused on developing analogues of the Galvin--Prikry and Ellentuck Theorems for topological spaces of subcopies of a given \Fraisse\ structure.
This line of inquiry was highlighted in
 Section 11 of  
\cite{Kechris/Pestov/Todorcevic05}, 
by Kechris, Pestov, and Todorcevic.
Specifically, they asked  for the development of infinite-dimensional Ramsey theory of the form $\bK\ra_*(\bK)^{\bK}_{\ell,t}$,
where $\bK$ is the \Fraisse\ limit of some \Fraisse\ class and $\ra_*$ means that the partitions of ${\bK\choose\bK}$ are required to be definable in some sense.

Identifying the universe of $\bK$ with $\mathbb{N}$,
 one can view the set of all  subcopies of $\bK$ as the subspace of the Baire space corresponding to the set of universes of subcopies of $\bK$.
In \cite{DobrinenRado19},
the author proved an infinite-dimensional Ramsey theorem for certain topological spaces of subcopies of the Rado graph. 
At the 2019 Luminy Workshop on Set Theory,
Todorcevic  asked the author  whether the infinite-dimensional Ramsey theorem would directly recover the exact big Ramsey degrees of the Rado graph.
The  approach in \cite{DobrinenRado19}
directly recovers exact big Ramsey degrees for vertex, edge, and non-edge colorings, but does not directly recover exact big Ramsey degrees for most graphs  with three or more vertices. 
Thus,   the first motivation for this paper was to develop infinite-dimensional Ramsey theory for the Rado graph which would directly recover known exact big Ramsey degrees from a  Nash-Williams style corollary. 
This is done in Corollary \ref{cor.NW}.

The second motivation for this paper was to develop infinite-dimensional Ramsey theory for a large collection of \Fraisse\ structures  for which exact big Ramsey degrees are already known, thus making progress on the question of Kechris, Pestov, and Todorcevic discussed above. 
The Main Theorem of this paper, Theorem \ref{thm.main},
develops infinite-dimensional Ramsey theory for all \Fraisse\ structures with finitely many relations of arity at most two satisfying  amalgamation properties called SDAP$^+$ and  LSDAP$^+$, developed by Coulson, Dobrinen, and Patel  in \cite{CDP1} and \cite{CDP2} to prove exact big Ramsey degrees with a simple characterization in terms of diagonal antichains 
in coding trees of $1$-types (see Theorem \ref{thm.CDP}).
The class of homogeneous structures satisfying 
SDAP$^+$ includes the rationals and  the rationals with an equivalence relation with finitely many dense equivalence classes, as well as
the Rado graph, 
 generic $n$-partite graphs, the generic tournament and digraph,  more generally unrestricted structures with finitely many binary relations, as well as 
versions of these with an additional linear order forming a dense linear order on the \Fraisse\ limit.
The class of LSDAP$^+$ structures includes the rationals with a convex equivalence relation and a natural 
hiearchy of such structures with successively coarser convex equivalence relations.
(See Section 5 of \cite{CDP2} for a catalogue of \SDAP$^+$ and   \LSDAP$^+$ structures and their big Ramsey degree results.)

The infinite-dimensional Ramsey theorem in this paper recovers the  big Ramsey degrees
proved in \cite{CDP2}
 for SDAP$^+$ and LSDAP$^+$  structures   in the following manner:
For each diagonal antichain $A$ representing a  finite substructure $\bfA$ of $\bK$,
Corollary \ref{cor.NW} shows that 
given any finite coloring of the copies of $\bfA$ in $\bK$,
 there is a subcopy of $\bK$ in which all copies of $\bfA$ represented by the similarity type of $A$ have the same color.
Note  that 
 the lower bound argument in \cite{CDP2}  showing that each diagonal antichain representing $\bfA$ persists in each subcopy of $\bK$  does  not follow from   the infinite-dimensional Ramsey theory in this paper. 
Rather, given that result, 
we  can conclude that Corollary \ref{cor.NW}  recovers  exact big Ramsey degrees.

We remark  on the necessity (in one form or another) of  diagonal coding antichains  and similarity types for   infinite dimensional Ramsey theory.
In the context of Countable Choice, one can  well-order the vertices of a countably infinite structure $\bK$ in order type $\om$; that is, we may assume that the  universe of $\bK$ is $\bN$.
Since our language is countable, we can linearly order its symbols.
Taken together, these automatically induce a coding tree of $1$-types representing $\bK$.
Big Ramsey degrees of \Fraisse\ structures  present   constraints for the development of infinite-dimensional  structural Ramsey theory.
A  \Fraisse\ limit  $\bK$
of a \Fraisse\ class $\mathcal{K}$ is said to have
 {\em finite big Ramsey degrees} if  for each $\bfA\in \mathcal{K}$, there is some positive integer $t$ such that  for each $\ell\ge 2$,
\begin{equation}
 \bK\ra (\bK)^{\bfA}_{\ell,t}.
 \end{equation}
 This  is the structural analogue of the infinite Ramsey Theorem \ref{thm.RamseyInfinite}.
  When such a $t$ exists for a given $\bfA$, using the conventions in  \cite{Kechris/Pestov/Todorcevic05},  $T(\bfA,\mathcal{K})$ denotes the minimal such $t$ and is called the  {\em big Ramsey degree} of $\bfA$ in $\bK$.
In all known cases, the big Ramsey degree $T(\bfA,\bK)$ corresponds to a  canonical partition
  of  ${\bK\choose\bA}$ (the copies of $\bA$ in $\bK$)
   into $T(\bA,\bK)$ many pieces
   each of which is
{\em persistent}, meaning that for any member $\bM$ of
  ${\bK\choose\bK}$, the set
   ${\bM\choose\bA}$ meets every piece in the partition.
Through the view of coding trees of $1$-types, canonical partitions for \SDAP$^+$ and \LSDAP$^+$ structures are characterized by finite  diagonal coding  antichains 
(see Theorem \ref{thm.CDP} and preceding definitions).
It is useful to  think of finite  big Ramsey degrees as  a structural Ramsey  theorem  
where one finds an expanded structure  which guarantees one color for all copies of $\bA$ in that expansion. 
Big Ramsey degrees of size two or more present 
a fundamental constraint to the development of infinite-dimensional structural  Ramsey theory:
any infinite-dimensional theorem must   restrict to 
a subspace of 
${\bK\choose\bK}$ where all members  have the same similarity type in the coding tree of $1$-types.

Given any \Fraisse\ structure  $\bK$ satisfying \SDAP$^+$ or \LSDAP$^+$ with universe $\mathbb{N}$, and given a subcopy $\bM$ of $\bK$,
let $\bK(\bM)$ denote the collection of all 
$\bN
\in {\bM\choose\bK}$ with the same (induced)  similarity type as $\bM$ has as an enumerated substructure of $\bK$.
(This space will be precisely defined  in Section \ref{sec.MainTheorem}.)
Note that $\bK(\bM)$ is a topological space,  identified with  the subspace of the Baire space consisting of the universes of all structures in $\bK(\bM)$.
 The following is the main theorem of the paper.

\begin{thmmain}
Let $\bK$ be an enumerated \Fraisse\ structure satisfying {\SDAP$^+$} (or {\LSDAP$^+$})
with finitely many relations of arity at most two,
and let $\bfD$ be a subcopy of $\bK$ such that the subtree $\bD$ of the coding tree of $1$-types over $\bK$  induced by the vertices in $\bfD$  is a good diagonal antichain. 
Then each Borel subset of $\bK(\bfD)$ is completely Ramsey, and hence Ramsey.
\end{thmmain}

From the methods used in this paper, we immediately obtain the following stronger Ellentuck analogue for certain structures.

\begin{thmEll}
Let $\bK$ be  any one of the following structures 
with universe $\mathbb{N}$:
The rationals, $\bQ_n$, $\bQ_{\bQ}$, and or any \Fraisse\  structure
satisfying {\SDAP$^+$}  (or {\LSDAP$^+$}) for which the coding tree of $1$-types  $\bU(\bK)$ 
has the property that on any given level of 
  $\bU(\bK)$, only the coding node splits.
Then the spaces $\mathcal{D}(\bD)$, where $\bD$ is a diagonal coding antichain for $\bK$, are actually topological Ramsey spaces.
\end{thmEll}

The paper is orginized as follows:
Background from  \cite{CDP1}  and \cite{CDP2} is presented in Section \ref{sec.background}.
Section \ref{sec.BaireSpaceDCA}
 defines the spaces of diagonal coding antichains representing subcopies of a given homogeneous structure $\bK$.
The pretext for our notation and set-up is  Chapter 5  of Todorcevic's 
book \cite{TodorcevicBK10} on topological Ramsey spaces.
Theorem \ref{thm.matrixHL} proves an
Extended Pigeonhole Principle, a strong version of Todorcevic's Axiom \bf A.4\rm.
Theorem \ref{thm.best} proves a Galvin--Prikry style theorem for spaces of diagonal coding antichains with a metric topology.
 Theorem  \ref{thm.main} in Section \ref{sec.MainTheorem}
interprets this back into natural subspaces of the Baire space, proving the main theorem of this paper.
Corollary \ref{cor.NW}  answers a question of Todorcevic, showing that the Nash-Williams style corollary of our main theorem recovers big Ramsey degrees.
Theorem \ref{thm.SDAPEllentuck} proves Ellentuck analogues for  structures which have a certain amount of rigidity.



\section{Background}\label{sec.background}

This section reviews  \Fraisse\ theory, amalgamation properties, and coding tree notions  from 
   \cite{CDP1}.


\subsection{\Fraisse\ theory and  substructure amalgamation properties}\label{subsec.background}
In this paper,
all languages 
$\mathcal{L}$  will consist of finitely many relation symbols $\{R_i:i< n\}$, with the arity $k_i$  of $R_i$ being either $1$ or $2$.
An {\em $\mathcal{L}$-structure}   is an object $\bfA=\lgl \mathrm{A}, R^{\bfA}_0,\dots, R^{\bfA}_{n-1}\rgl$, where 
 the {\em universe} of  $\bfA$, denoted by 
$\mathrm{A}$, is non-empty and 
$ R^{\bfA}_i\sse \A^{k_i}$.
Finite structures will  be denoted by $\bfA,\bfB, \bfC,\dots$, 
and  their universes  by $\A,\B,\C,\dots$.  
Infinite structures will typically be denoted by $\bK,\bM,\bN,\dots$, and their universes
by $\K,\mathrm{M},\mathrm{N}, \dots$. 
The elements of the universe of a structure will be called {\em vertices}.

With no loss of generality, 
we make the following assumptions:
(a) $\mathcal{L}$ has at least one unary relation symbol.
(b)
Any  structure $\bfA\in\mathcal{K}$ has the property  that  each vertex $v\in \mathrm{A}$ satisfies $R^{\bfA}(v)$ for exactly one unary relation symbol  $R$ in $\mathcal{L}$.
(c)
For each unary relation symbol $R\in\mathcal{L}$,
there is some $\bfA\in\mathcal{K}$ and a 
vertex $v\in \mathrm{A}$  such that  $R^{\bfA}(v)$.
We say that the unary relations  are {\em non-trivial} exactly when  
$\mathcal{L}$ 
 has two or more unary relation symbols.

For $\mathcal{L}$-structures $\bfA$ and $\bfB$,
an {\em embedding}  $e:\bfA\ra \bfB$  is an injection
on their universes $e:\mathrm{A}\ra \mathrm{B}$ with the property that
for all $i<n$, 
\begin{equation*}
R_i^{\bfA}(a_1,\dots,a_{n_i})\longleftrightarrow R_i^{\bfB}(e(a_1),\dots,e(a_{n_i}))
\end{equation*}
The $e$-image of $\bfA$ is  called  a {\em copy} of $\bfA$ in $\bfB$.
If $e$ is the identity map, then
 $\bfA$ is a {\em substructure} of $\bfB$.
If $e$  is onto $\bfB$ then $e$ is an 
 {\em isomorphism} and we say that $\bfA$ and $\bfB$ are {\em isomorphic}.
We write $\bfA\le\bfB$  exactly when 
 there is an embedding of  $\bfA$  into $\bfB$, and
 $\bfA\cong\bfB$  exactly when   there is an isomorphism from $\bfA$ onto $\bfB$.

A class $\mathcal{K}$ of finite structures
 is called a {\em \Fraisse\ class}  if it is
nonempty,  closed under isomorphisms,
 hereditary, and satisfies the joint embedding and amalgamation properties.
$\mathcal{K}$ is  {\em hereditary} if whenever $\bfB\in\mathcal{K}$ and  $\bfA\le\bfB$, then also $\bfA\in\mathcal{K}$.
 $\mathcal{K}$ satisfies the {\em joint embedding property} if for any $\bfA,\bfB\in\mathcal{K}$,
there is a $\bfC\in\mathcal{K}$ such that $\bfA\le\bfC$ and $\bfB\le\bfC$. 
 $\mathcal{K}$ satisfies the {\em amalgamation property} if for any embeddings
$f:\bfA\ra\bfB$ and $g:\bfA\ra\bfC$, with $\bfA,\bfB,\bfC\in\mathcal{K}$,
there is a $\bfD\in\mathcal{K}$ and  there are embeddings $r:\bfB\ra\bfD$ and $s:\bfC\ra\bfD$ such that
$r\circ f = s\circ g$.
 Note that in a finite relational language, there are only countably many finite structures up to isomorphism.

A \Fraisse\ class $\mathcal{K}$ satisfies the 
{\em disjoint amalgamation property} (DAP)  if given $\bfA,\bfB,\bfC\in\mathcal{K}$ and  embeddings $e:\bfA\ra \bfB$ and $f:\bfA\ra \bfC$,
there is some $\bfD\in\mathcal{K}$ and embeddings $e':\bfB\ra\bfD$ and $f':\bfC\ra\bfD$ such that $e'\circ e=f'\circ f$, and $e'[B]\cap f'[C]=e'\circ e[A]=f'\circ f[A]$.
The DAP  is often  called  the {\em strong amalgamation property}  and is
 equivalent  to the {\em strong embedding property},
 which
says that
 for any $\bfA\in\mathcal{K}$, $v\in\mathrm{A}$, and embedding $\varphi:(\bfA-v) \ra\bK$,
there are infinitely many different extensions of $\varphi$ to embeddings of $\bfA$ into $\bK$.  (See \cite{CameronBK90}.)
We say that  $\mathcal{K}$ satisfies the {\em free amalgamation property} (FAP) if it satisfies the DAP and moreover, $\bfD$ can be chosen so that $\bfD$ has no additional relations other than those inherited from $\bfB$ and $\bfC$.

The following  amalgamation properties,  SFAP and SDAP, were first formulated in  \cite{CDP1}.

\begin{defn}[\cite{CDP1}]\label{defn.SFAP}
A \Fraisse\ class $\mathcal{K}$ has the
{\em  \SFAPfull\ (\SFAP)} if $\mathcal{K}$ has free amalgamation
and given  $\bfA,\bfB,\bfC,\bfD\in\mathcal{K}$, the following holds:
Suppose
\begin{enumerate}
\item[(1)]
$\bfA$  is a substructure of $\bfC$, where
 $\bfC$ extends  $\bfA$ by two vertices,
say $\mathrm{C}\setminus\mathrm{A}=\{v,w\}$;

\item[(2)]
 $\bfA$  is a substructure of $\bfB$ and
 $\sigma$ and $\tau$  are
 $1$-types over $\bfB$  with   $\sigma\re\bfA=\type(v/\bfA)$ and $\tau\re\bfA=\type(w/\bfA)$; and
\item[(3)]
$\bfB$ is a substructure of $\bfD$ which extends
 $\bfB$ by one vertex, say $v'$, such that $\type(v'/\bfB)=\sigma$.

\end{enumerate}
  Then there is
an   $\bfE\in\mathcal{K}$ extending  $\bfD$ by one vertex, say $w'$, such that
  $\type(w'/\bfB)=\tau$, $\bfE\re (
  \mathrm{A}\cup\{v',w'\})\cong \bfC$,
  and $\bfE$ adds no other relations over $\mathrm{D}$.
\end{defn}

\SFAP\ 
can be stated 
in terms of  embeddings, but 
its formulation  via   substructures  and $1$-types is more closely aligned with 
 its uses in  the forcing proof in Section \ref{sec.APP}.
In \cite{CDP1} it was remarked that 
\SFAP\ is equivalent to free amalgamation along with
 a model-theoretic property that may be termed {\em free 3-amalgamation},
which is 
 a special case of the {\em disjoint 3-amalgamation} property defined in \cite{Kruckman19}.
Kruckman showed  in \cite{Kruckman19} that if the age of a \Fraisse\ limit $\bK$ has disjoint amalgamation and disjoint 3-amalgamation, then $\bK$ exhibits a model-theoretic tameness property called  \emph{simplicity}.

The next amalgamation property extends \SFAP\ to disjoint amalgamation classes.

\begin{defn}[\cite{CDP1}]\label{defn.SDAP}
A \Fraisse\ class $\mathcal{K}$ has the
{\em  \SDAPfull}  (\SDAP)
 if $\mathcal{K}$
has disjoint amalgamation,
and the following holds:
Given   $\bfA, \bfC\in\mathcal{K}$, suppose that
 $\bfA$ is a substructure of $\bfC$, where   $\bfC$ extends  $\bfA$ by two vertices, say $v$ and $w$.
Then there exist  $\bfA',\bfC'\in\mathcal{K}$, where
$\bfA'$
contains a copy of $\bfA$ as a substructure
and
$\bfC'$ is a disjoint amalgamation of $\bfA'$ and $\bfC$ over $\bfA$, such that
letting   $v',w'$ denote the two vertices in
 $\mathrm{C}'\setminus \mathrm{A}'$ and
assuming (1) and (2), the conclusion holds:
 \begin{enumerate}
 \item[(1)]
Suppose
$\bfB\in\mathcal{K}$  is any structure
 containing $\bfA'$ as a substructure,
and let
 $\sigma$ and $\tau$  be
  $1$-types over $\bfB$  satisfying    $\sigma\re\bfA'=\type(v'/\bfA')$ and $\tau\re\bfA'=\type(w'/\bfA')$.
\item[(2)]
Suppose
$\bfD\in \mathcal{K}$  extends  $\bfB$
by one vertex, say $v''$, such that $\type(v''/\bfB)=\sigma$.
\end{enumerate}
Then
  there is
an  $\bfE\in\mathcal{K}$ extending   $\bfD$ by one vertex, say $w''$, such that
  $\type(w''/\bfB)=\tau$ and  $\bfE\re (
  \mathrm{A}\cup\{v'',w''\})\cong \bfC$.
\end{defn}

It is straightforward to see that 
\SDAP\ implies \SFAP\ (let $\bfA'=\bfA$ and $\bfC'=\bfC$),
and that 
 \SFAP\ and \SDAP\ are each
preserved  under free superposition.
These two amalgamation properties  were formulated by extracting   properties 
of  \Fraisse\ classes for which  the forcing partial order in   Theorem \ref{thm.matrixHL}
can just  be  extension,  
from which  simple characterizations of their big Ramsey degrees in \cite{CDP2} follow.


\subsection{Coding trees of $1$-types}\label{subsec.ct}

This subsection reproduces  notions from \cite{CDP1} which will be used throughout this  paper. 
Given a \Fraisse\ class $\mathcal{K}$,  an {\em enumerated \Fraisse\ structure} is a \Fraisse\ limit  $\bK$
of $\mathcal{K}$ with  universe $\mathbb{N}=\{ 0,1,2,\dots\}$.
For the sake of clarity, we  use $v_n$  (rather than $n$) to denote the  $n$-th vertex of $\bK$. 
We   let $\bK_n$ denote $\bK\re \{v_i:i<n\}$, the restriction of $\bK$ to its first $n$ vertices.
All types will be quantifier-free 1-types, with variable $x$, over some finite initial segment of $\bK$; 
the notation ``tp'' denotes a complete quantifer-free 1-type.
For $n \ge 1$, a type over $\bK_n$ must contain the formula
$\neg (x = v_i ) $ for each $i < n$.
Given a type $s$ over $\bK_n$, for any $i < n$, $s \re \bK_i$ denotes the restriction of $s$ to parameters from $\bK_i$.

\begin{defn}[The  Coding  Tree of $1$-Types, \cite{CDP1}]\label{defn.treecodeK}
The {\em coding  tree of $1$-types}
$\bS(\bK)$
for an enumerated \Fraisse\ structure $\bK$
 is the set of all complete
  $1$-types over initial segments of $\bK$
along with a function $c:\mathbb{N}\ra \bS(\bK)$ such that
$c(n)$ is the
$1$-type
of $v_n$
over $\bK_n$.
The tree-ordering is  inclusion.
\end{defn}

We will often  write $\bS$ and  $c_n$ in place of $\bS(\bK)$ and  $c(n)$, respectively.
Let  $\bS(n)$
denote
 the collection
 of all   $1$-types  $\type(v_i/\bK_n)$, where
 $i\ge n$, and note that
 $c_n$ is a node in $\bS(n)$.
The set  $\bS(0)$ consists of
 the
 $1$-types over the empty structure $\bK_0$.
A {\em level set} is a subset $X\sse\bS(n)$  for some $n$.
For $s\in \bS(n)$, the immediate successors of $s$   are  exactly those $t\in \bS(n+1)$ such that $s\sse t$.
Each  set $\bS(n)$  is finite, since the language
$\mathcal{L}$  has  finitely many finitary relation symbols.

A node $s\in\bS(n)$ has {\em length}  $n+1$, denoted  by $|s|$, and uniquely induces
  the sequence
$\lgl s(i):i< |s|\rgl$ defined as follows:
  $s(0)$  denotes the set of formulas in $s$ involving
no parameters, and 
for  $1\le i< |s|$,
 $s(i)$ denotes  the set of those formulas in $s\re \bK_i$
in which $v_{i-1}$ appears as the parameter.
For $j< |s|$, note that $\bigcup_{i\le j}s(i)$  is the  predecessor of $s$ in 
$\bS(j)$.
For $\ell\le |s|$,
we let  $s \re \ell$  denote $\bigcup_{i<\ell}s(i)$.
Given $s,t\in \bS$,
$s\wedge t$ denotes 
the {\em meet} of $s$ and $t$,  which is 
$s \re \bK_m$ 
where $m$ is maximal such that
 $s \re \bK_m=t \re \bK_m$.

Let $\Gamma$ denote $\bS(0)$,  the set of  complete $1$-types over the empty set that are realized in $\bK$. For $\gamma\in\Gamma$, we write ``$\gamma(v_n)$ holds in $\bK$'' when $\gamma$ is the 1-type of $v_n$ over the empty set.
The following modification of
 Definition \ref{defn.treecodeK}  of $\bS(\bK)$  is useful  for \Fraisse\ classes which have both non-trivial unary relations and a linear order.

\begin{defn}[The Unary-Colored  Coding Tree of $1$-Types,
\cite{CDP1}]\label{defn.ctU}
Let $\mathcal{K}$ be a \Fraisse\ class in language 
$\mathcal{L}$ and $\bK$ be an enumerated \Fraisse\ structure for $\mathcal{K}$.
For $n\in\mathbb{N}$,
let $c_n$ denote the $1$-type
of $v_n$
over $\bK_n$
(exactly as in the definition of $\bS(\bK)$).
Let $\mathcal{L}^-$ denote
the collection of all relation symbols in $\mathcal{L}$ of arity greater than one,
and let $\bK^-$ denote the reduct of $\bK$ to $\mathcal{L}^-$
and $\bK_n^-$ the reduct of $\bK_n$ to $\mathcal{L}^-$.

The {\em $n$-th level}, denoted  $\bU(n)$,  is 
the collection
of all $1$-types
$s$ over $\bK^-_n$
in the language $\mathcal{L}^-$
such that
for some $i\ge n$,
$v_i$ satisfies $s$.
Define $\bU=\bU(\bK)$
 to be
$\bigcup_{n<\om}\bU(n)$.
The tree-ordering on $\bU$ is simply inclusion.
The {\em unary-colored coding tree of $1$-types}
is
the tree $\bU$ along with the  function $c:\om\ra \bU$ such that $c(n)=c_n$.
Thus,
$c_n$ is the $1$-type
(in the language $\mathcal{L}^-$) of $v_n$
 in $\bU(n)$  along with the additional ``unary color''
$\gamma\in\Gamma$ such that
$\gamma(v_n)$ holds in $\bK$.
 \end{defn}

\begin{rem}
 If $\mathcal{L}$ has no non-trivial unary relation symbols
 then 
$\bU(\bK)=\bS(\bK)$.
If  $\mathcal{K}$ satisfies \SFAP,
it suffices to work in $\bS(\bK)$.
The purpose of $\bU(\bK)$ is to handle  cases
such as $\bQ_n$, the rationals with an equivalence relation with finitely many equivalence classes each of which is dense in the rationals.
\end{rem}


\subsection{Passing types and  similarity}\label{subsec.3.2}

This subsection reproduces definitions  from \cite{CDP1},
with 
 simplified versions given 
 due to the fact that all    relation symbols in this article have  arity at most two.
Throughout, fix $\bK$ and let $\bS$ denote $\bS(\bK)$.
All of the  instances of 
$\bS$ 
in this subsection 
may be substituted with $\bU:=\bU(\bK)$.

\begin{defn}[Passing Type, \cite{CDP1}]\label{defn.passingtype}
Given $s,t\in \bS$  with $|s|<|t|$,
we call
  $t(|s|)$
the {\em passing type of $t$  at $s$}.
We also call $t(|s|)$
the {\em passing type of $t$ at $c_n$},
where $n$ is the integer such that $|c_n|=|s|$.
\end{defn}

Note that passing types are  partial $1$-types which contain only binary relation symbols.

\begin{defn}[Similarity of Passing Types, \cite{CDP1}]\label{defn.prespt}
Let  $m,n\in\mathbb{N}$, and 
   let
$f: \{v_m,x\} \to \{v_n,x\}$ be  given  by
$f(v_m)=v_n$ and 
 $f(x) = x$.
Suppose  $s,t\in\bS$  are such that 
$|c_m|<|s|$ and $|c_n|<|t|$.
We write
\begin{equation}
s(c_m)\sim t(c_n)
\end{equation}
when,
given any
relation symbol $R\in \mathcal{L}$ of arity two and 
ordered pair $(z_0,z_1)$
such that $\{z_0,z_1\}=\{v_m,x\}$,
it follows that 
  $R(z_0,z_1)\in s(c_m)$ if and only if
  $R(f(z_0),f(z_1))\in t(c_n)$.
   When $s(c_m)\sim t(c_n)$ holds,
 we say that  the passing type of  $s$ at $c_m$  is {\em similar} to the passing type of $t$ at $c_n$.
\end{defn}

It is clear that $\sim$ is an equivalence relation.

\begin{defn}\label{defn.similarityoftypes}
Let $\bfA$, $\bfB$ be  finite  substructures of $\bK$
with universes $\lgl v_{j_i}:i<n\rgl$, $\lgl  v_{k_i}:i<n\rgl$, respectively. 
Let 
 $s\in \bS(\ell)$ and $t\in \bS(\ell')$, where 
$\ell\ge |c_{j_{n-1}}|+1$ and $\ell'\ge |c_{k_{n-1}}|+1$. 
We say that  $s\re \bfA$ and $t\re\bfB$ are  {\em similar}, and write $s\re \bfA\sim   t\re\bfB$, if and only if 
for each $i<n$, $s(c_{j_i})\sim t(c_{k_i})$.
\end{defn}

\begin{fact}[\cite{CDP1}]\label{fact.simsamestructure}
Let $A=\lgl v_{j_i}:i<n\rgl$ and $B=\lgl v_{k_i}:i<n\rgl$ be  sets of vertices in $\bK$,
and let 
$\bfA:=\bK\re A$ and $\bfB:=\bK\re B$. 
Then $\bfA$ and $\bfB$ 
are isomorphic as ordered substructures of $\bK$,
if and only if
\begin{enumerate}
\item
$c_{j_i}$ and $c_{k_i}$ contain
the same parameter-free formulas, for each $i<n$;
and
\item
$c_{j_i}\re (\bK\re \{v_{j_m}:m<i\})
\sim
c_{k_i}\re (\bK\re \{v_{k_m}:m<i\})$,
 for all $i<n$.
\end{enumerate}
\end{fact}

A lexicographic order $\prec$ on $\bS$ is induced by 
fixing a linear ordering on the relation symbols  in $\mathcal{L}$ and their negations. 
We may assume that the
negated relation symbols  appear in the linear order before the  relation symbols.
Since any node of $\bS$ is completely determined by such atomic and negated atomic formulas, this
lexicographic order gives rise to a linear order on $\bS$, which we again denote by $\prec$, with the following properties:
If 
 $s \subsetneq t$,
then $s\prec t$.
For any incomparable   $s,t \in\bS$, if
 $|s\wedge t|=n$,
 then
$s\prec t$ if and only if $s\re(n+1)\prec t\re (n+1)$.
This order $\prec$
generalizes
the lexicographic order  for the case of binary relational structures in \cite{Sauer06}, \cite{Laflamme/Sauer/Vuksanovic06}, \cite{DobrinenJML20}, \cite{DobrinenH_k19}, and \cite{Zucker20}.

Given  $S\sse\bS$, let $\lgl c^S_i:i<n\rgl$ enumerate the coding nodes in $S$ in increasing length, where $n\in \mathbb{N}\cup\{\mathbb{N}\}$.

\begin{defn}[Similarity Map, \cite{CDP1}]\label{def.ssmap}
Let  $S$ and $T$  be meet-closed subsets
of $\bS$.
A function $f:S\ra T$ is a {\em similarity map} of $S$ to $T$ if 
$f$ is a bijection and 
for all nodes
$s, t \in \bS$,
the following hold:
\begin{enumerate}
\item
$f$ preserves $\prec$:
$s\prec t$ if and only if $f(s)\prec f(t)$.

\item
$f$ preserves meets, and hence splitting nodes:
$f(s\wedge t)=f(s)\wedge f(t)$.

\item
$f$ preserves relative lengths:
$|s|<|t|$ if and only if
$|f(s)|<|f(t)|$.

\item
$f$ preserves initial segments:
$s\sse t$ if and only if $f(s)\sse f(t)$.

\item
$f$ preserves  coding  nodes and their
parameter-free formulas:
Given  $c_i^S\in S$, then  $f(c_i^S)=c^T_i$;
moreover,
for  $\gamma\in\Gamma$,
$\gamma(v^S_i)$ holds in $\bK$
  if and only if
 $\gamma(v^T_i)$ holds in $\bK$,
 where $v^S_i$ and $v^T_i$ are the vertices of $\bK$ represented by coding nodes
 $c^S_i$ and $c^T_i$, respectively.
\item
$f$ {\em preserves relative passing types}:
$s(c^S_i)\sim f(s)(c^T_i)$, for all coding nodes $c^S_i$ in $S$.
\end{enumerate}

When there is a similarity map between $S$ and $T$, we say that $S$ and $T$  are {\em similar} and  write $S\sim T$.
Given a subtree $S$ of $\bS$,
let $\Sim(S)$ denote the collection of all subtrees  $T$ of $\bS$ which are similar to $S$.
If $T'\sse T$ and $f$ is a  similarity map of $S$ to $T'$, then   $f$ is  called a {\em similarity embedding} of $S$ into $T$.
\end{defn}


\subsection{The properties SDAP$^+$ and LSDAP$^+$}\label{subsec.3.3}

The property SDAP$^+$ is a strengthening of SDAP, and LSDAP$^+$ is a `labeled' version of SDAP$^+$.

\begin{defn}[Subtree]\label{defn.subtree}
Let $T$ be a subset of   $\bS$, and let $L$ be the set of lengths of
coding nodes in $T$
and lengths of meets
of two incomparable nodes (not necessarily coding nodes) in $T$. Then $T$
is a {\em subtree} of $\bS$
if $T$ is closed under meets and
closed under initial segments with lengths in
$L$.
\end{defn}

\begin{defn}[Diagonal tree]\label{def.diagskew}
A  subtree $T\sse \bS$ is 
{\em diagonal}
if each level of $T$ has at most one splitting node,
each splitting node in $T$ has degree two (exactly two immediate successors), and
coding node levels in $T$ have no splitting nodes.
\end{defn}

Subtrees of $\bS$  correspond to   substructures of $\bK$,  and vice versa for diagonal subtrees:
Given a subtree $T\sse\bS$, let
$N^T$ denote the set of natural numbers $n$ such that 
$c_n\in T$; 
 let $\bK\re T$ denote the substructure of $\bK$ on universe $\{v_n:n\in N^T\}$.
 We call $\bK\re T$ the
 {\em substructure of $\bK$
 represented by the coding nodes in $T$}, or simply {\em  the substructure represented by $T$}.
In the reverse direction, given a substructure $\bM\le \bK$,
 let $\bS\re\bM$ denote the subtree of $\bS$
induced by the meet-closure of the coding nodes $\{c_n:v_n\in \mathrm{M}\}$, and 
 call $\bS\re \bM$ the {\em subtree of $\bS$ induced by $\bJ$}.
 If $T$ is a diagonal subtree and 
$\bM=\bK\re T$,
 then $\bS\re \bM=T$,
 as $T$ being  diagonal ensures that the  coding nodes in $\bS\re \bM$ are exactly those in  $T$.

Given two
substructures  $\bfA,\bfB$ of $\bK$, we write $\bfA\cong^{\om} \bfB$
when there exists an $\mathcal{L}$-isomorphism between $\bfA$ and $\bfB$ that preserves
 the linear order on
their universes.
It follows from
 Fact \ref{fact.simsamestructure} that 
 for any subtrees
$S,T\sse \bS$,
$S\sim T$ implies that $\bK\re S\cong^{\om}\bK\re T$.

\begin{notation}\label{notn.cong<}
Given a diagonal subtree $T$,
 let 
$\lgl c^T_n: n<N\rgl$  denote the
enumeration of   the coding nodes in $T$ in order of increasing length, and let
 $\ell^T_n$ denote  $|c^T_n|$, the {\em length} of $c^T_n$.
For a  finite subset $A\sse T$,  let
\begin{equation}
 \ell_A=\max\{|t|:t\in A\}\mathrm{\ \ and\ \ } \max(A)=
\{s\in A: |s|=\ell_A\}.
 \end{equation}
For any subset $A\sse T$, 
let
\begin{equation}
A\re \ell=\{t\re \ell : t\in A\mathrm{\ and\ }|t|\ge \ell\}
\end{equation}
 and let
\begin{equation}
A\rl \ell=\{t\in A:|t|< \ell\}\cup A\re \ell.
\end{equation}
Thus, $A\re \ell$ is a level set, while $A\rl \ell$ is the set of nodes in $A$ with length less than $\ell$ along with the truncation
to $\ell$ of the  nodes in $A$ of length at least
 $\ell$.
For  $A,B\sse T$,   $B$ is an {\em initial segment} of $A$ if   $B=A\rl \ell$
 for some $\ell$ equal to
   the length of some node in $A$.
   In this case, we also say that
   $A$ {\em end-extends} (or just {\em extends}) $B$.
If $\ell$ is not the length of any node in $A$, then
  $A\rl \ell$ is not a subset  of $A$, but  is  a subset of $\widehat{A}$, where
  $\widehat{A}$ denotes $\{t\re n:t\in A\mathrm{\ and\ } n\le |t|\}$.
Given  $t\in T$ at the level of a coding node in $T$, $t$ has exactly one immediate successor in $\widehat{T}$, which   we denote by 
$t^+$.
\end{notation}

Given a substructure $\bM$ of $\bK$,
let $\bM_n$ denote $\bM$ restricted to its first $n$ vertices. 
A tree  $T$ is  {\em perfect}  if   each node in  $T$ has  at least  two  incomparable extensions in $T$.

\begin{defn}[Diagonal Coding Subtree]\label{defn.sct}
A subtree $T\sse\bS$ is called a {\em diagonal coding subtree} if $T$ is diagonal, perfect, 
   $\bM:=\bK\re T\cong\bK$, and the following holds:
  \begin{enumerate}
 \item[]
Suppose  $s\in T$ 
with $|s|=\ell^T_{i-1}+1$ for some $i\ge 1$, or else suppose $s$ is 
the stem of $T$ and let $i=0$.
Then 
for each $n>i$ and  each
 $1$-type $\tau$ over $\bK_n$ such that 
$\tau\re\bK_i\sim s$,
there is  a $t\in T\re (\ell^T_{n-1}+1)$ extending $s$ such that $t\re \bM_n\sim \tau$.
\end{enumerate}
\end{defn}

\begin{defn}[Diagonal Coding Tree Property, \cite{CDP1}]\label{defn.DCTP}
A \Fraisse\ class $\mathcal{K}$ in language $\mathcal{L}$  satisfies the {\em Diagonal Coding Tree Property (DCTP)}
if  given any enumerated \Fraisse\ structure $\bK$
for $\mathcal{K}$,
there is a diagonal coding
subtree.
\end{defn}

\begin{defn}[$+$-Similarity, \cite{CDP1}]\label{def.plussim}
Let $T$ be a diagonal coding tree for
the \Fraisse\ limit $\bK$ of
a \Fraisse\
class $\mathcal{K}$, and
suppose $A$ and $B$
are finite subtrees of $T$.
We write  $A\plussim B$ and say that
 $A$ and $B$ are
 {\em $+$-similar} if and only if
  $A\sim B$ and
 one of the following two cases holds:
 \begin{enumerate}
 \item[]
   \begin{enumerate}
\item[\bf Case 1.]
 If $\max(A)$  has a splitting node in $T$,
 then so does $\max(B)$,
  and the similarity map from $A$ to $B$  takes the splitting node in $\max(A)$ to the splitting node in $\max(B)$.
    \end{enumerate}
     \end{enumerate}
       \begin{enumerate}
    \item[]
        \begin{enumerate}
  \item[\bf Case 2.]
If $\max(A)$ has a coding node, say
 $c^A_n$,
and  $f:A\ra B$ is  the similarity map,
then
  $s^+(n)\sim f(s)^+(n)$ for each $s\in \max(A)$.
 \end{enumerate}
    \end{enumerate}

Note that $\plussim$ is an  equivalence relation, and  $A\plussim B$ implies $A\sim B$.
When $A\sim B$ ($A\plussim B$), we say that they have the same {\em similarity type} ({\em $+$-similarity type}).
\end{defn}

\begin{defn}[Extension Property,  \cite{CDP1}]\label{defn.ExtProp}
We say that
$\bK$
has the {\em Extension Property} when
the following condition  (EP) holds:
\begin{enumerate}
\item[(EP)]
Suppose $A$ is a finite or infinite subtree of   some
$T\in\mathcal{T}$.
Let $k$ be given  and suppose
$\max(r_{k+1}(A))$ has a splitting node.
Suppose that $B$ is a $+$-similarity copy of $r_k(A)$ in $T$.
Let  $u$ denote the splitting node in $\max(r_{k+1}(A))$,
and let
 $s$ denote  the node in $\max(B)^+$ which must be extended to a splitting node in order to obtain a $+$-similarity copy of $r_{k+1}(A)$.
If $s^*$ is a splitting node  in $T$  extending $s$,
then  there are extensions of the rest of the nodes in $\max(B)^+$ to the same length as $s^*$ resulting in a $+$-similarity copy
of $r_{k+1}(A)$ which
can be extended to a copy of $A$.
\end{enumerate}
\end{defn}

It was shown in Lemma 4.17 of \cite{CDP1} that SFAP implies  the Extension Property, and that moreover,
any SFAP class with an additional linear order also easily satisfy the Extension Property.

\begin{defn}[\SDAP$^+$, \cite{CDP1}]\label{def.SDAPCodingTree}
A \Fraisse\ class $\mathcal{K}$ satisfies {\em \SDAP$^+$} if and only if $\mathcal{K}$ satisfies SDAP and 
any 
\Fraisse\ limit $\bK$ of $\mathcal{K}$ with universe $\mathbb{N}$
satisfies
 the Diagonal Coding Tree Property and 
the Extension Property.
\end{defn}

\begin{rem}
If  there exists an enumerated  \Fraisse\ limit $\bK$ of $\mathcal{K}$  satisfying DCTP and EP, then 
every enumerated \Fraisse\ limit of $\mathcal{K}$ also satisfies DCTP and EP. 
Thus, the property SDAP$^+$ is truly a property of the \Fraisse\ class $\mathcal{K}$.
This being the case, we will  refer to both a \Fraisse\ class and its \Fraisse\ limit as satisfying SDAP$^+$.
\end{rem}

The Labeled Substructure Disjoint Amalgamation Property is a labeled  version applicable to structures such as $\bQ_{\bQ}$, $\bQ_{{\bQ}_{\bQ}}$, and so forth (see Subsection 4.4 in \cite{CDP1}).
Since the proofs for SDAP$^+$ and LSDAP$^+$ structures are almost identical, for the sake of space, we will provide proofs  for  SDAP$^+$ structures.


Note that any \SDAP$^+$ and \LSDAP$^+$ structures can be encoded inside $\bU$ according to the following convention.

\begin{convention}\label{conv.unaries_partition}
There is some $k\ge 1$, a partition $P_i$ ($i<k$) of the unary relation symbols in $\mathcal{L}$  and  subtrees $T_i\sse\bU$ so that  the following hold:
\begin{enumerate}
\item
$T:=\bigcup_{i<k}T_i$
 forms a $k$-rooted diagonal coding tree;
\item
For each $i<k$,
the coding nodes in $T_i$ have only unary relations from $P_i$, and all the unary relation symbols from $P_i$ occur densely in $T_i$;
\item
Property (2) persists in every coding subtree of $T$. 
\end{enumerate}
In this way,  the partition $P_i$, $i<k$, is optimal and persistent. 
\end{convention}

For simplicity, though, we will assume the following convention.

\begin{convention}\label{conv.Gamma_ts}
Let $\mathcal{K}$ be a \Fraisse\ class in a language $\mathcal{L}$ and $\bK$ a \Fraisse\ limit of $\mathcal{K}$.
If
(a)
$\mathcal{K}$  satisfies \SFAP,
or
(b)
$\bK$
 satisfies \EEAP$^+$ and either
 has no unary relations or has no transitive relations, then
we  work inside a diagonal coding subtree $\bT$ of
$\bS$.
Otherwise, we work inside a diagonal coding subtree $\bT$ of $\bU$.
 \end{convention}

Convention 
\ref{conv.Gamma_ts}
is a special case of Convention 
\ref{conv.unaries_partition}.
We will prove the main theorems assuming Convention \ref{conv.Gamma_ts},
noting that it is straightforward 
 to recover the general results under 
Convention 
\ref{conv.unaries_partition}.


We conclude this section with the  result on big Ramsey degrees  from \cite{CDP2}.
A  set of  incomparable coding nodes  $A\sse \bU$ is 
 called an {\em  antichain}.
An antichain of coding nodes in $\bU$ is called a {\em diagonal antichain}
if the tree induced by the meet closure of the antichain is diagonal.

\begin{defn}[Diagonal Coding Antichain, \cite{CDP1}]\label{defn.DCA}
A diagonal antichain $A\sse\bU$ is called a 
{\em diagonal coding antichain} (DCA)
if $\bK\re A\cong \bK$ and (the tree induced by)
$A$ is a Diagonal Coding Subtree.
\end{defn}

\begin{lem}[\cite{CDP1}]\label{lem.bD}
Suppose $\mathcal{K}$ is a \Fraisse\ class  satisfying
\SDAP$^+$ or \LSDAP$^+$.
Then  there is an infinite  diagonal antichain of coding nodes  $\bD\sse \bU$  so that $\bK\re \bD\cong^{\om}\bK$.
\end{lem}

\begin{thm}[Coulson, Dobrinen, Patel, \cite{CDP2}]\label{thm.CDP}
Let $\mathcal{K}$ be a \Fraisse\ class satisfying {\SDAP$^+$}
or {\LSDAP$^+$}
 in a finite relational language with relation symbols of arity at most two.
Then for each finite structure $\bfA\in \mathcal{K}$,
the big Ramsey degree of $\bfA$ in $\bK$ is exactly the number of  similarity types of diagonal antichains
representing a copy of $\bfA$.
\end{thm}


\section{Baire spaces of diagonal coding antichains}\label{sec.BaireSpaceDCA}

We now set up the subspaces of the Baire space 
 for which we will prove analogues of the Galvin--Prikry and Ellentuck Theorems.
Given a \Fraisse\ structure $\bK$
with universe $\mathbb{N}$,
each subcopy of $\bK$ can be identified  with 
its universe, an infinite subset of $\mathbb{N}$.
Thus, the collection  ${\bK\choose\bK}$  of all subcopies of $\bK$ is naturally  identified with a subspace of the Baire space $[\mathbb{N}]^{\mathbb{N}}$.

The existence of big Ramsey degrees  greater than one precludes 
 any 
simplistic approach to infinite-dimensional Ramsey theorems in terms only of definable  sets  on 
the full space
 ${\bK\choose\bK}$ of subcopies of $\bK$.
At the same time,
Theorem \ref{thm.CDP}
shows us where to look for 
 viable infinite-dimensional   Ramsey theorems: 
 precisely on subspaces of the Baire space
determined by 
collections of 
coding subtrees  which are all similar to each other. 
While infinite-dimensional Ramsey theorems for all  such spaces follow from the methods in this paper,
we will concentrate on  spaces of 
 antichains similar to some fixed diagonal coding antichain, as such  spaces will additionally 
 recover  exact   big Ramsey degrees.

\begin{defn}[Spaces of Diagonal Coding Antichains]\label{defn.BS}
Let $\mathcal{K}$ be a  \Fraisse\ class satisfying SDAP$^+$,  and 
 let $\bK$ be a \Fraisse\ limit of 
$\mathcal{K}$ with universe $\mathbb{N}$.
Let $\bD$ be any diagonal coding 
antichain;
that is,  $\bD$ is a diagonal antichain  of coding nodes  in $\bS(\bK)$ or $\bU(\bK)$  representing a subcopy of  $\bK$.
(Recall Convention \ref{conv.Gamma_ts}.)

Let  $\mathcal{D}(\bD)$ denote the  collection of all subsets $M\sse \bD$ such that $M\sim \bD$.
The partial ordering $\le$ on $\mathcal{D}(\bD)$ is simply inclusion.
For  $M\in\mathcal{D}(\bD)$,  when we write $N\le M$, it is implied that $N\in\mathcal{D}(\bD)$.
When $\bD$ is understood, we  simply write $\mathcal{D}$.

Each diagonal coding antichain $M\in\mathcal{D}$ 
uniquely determines the  substructure $\bfM:=\bK\re\{v_i:c_i\in M\}$.
Since $M\sim \bD$, it follows that $\bM\cong^{\om}\bK$.
Let $\bfD$ denote $\bK\re\bD$, and let 
\begin{equation}
\bK(\bfD)
=\{\bM:M\in\mathcal{D}\}.
\end{equation}
Then $\bK(\bfD)$ is  identified with  the  subspace $\{\{i\in\mathbb{N}:v_i\in \bM\}:\bM\in \bK(\bfD)\}$ of the Baire space.
These are the spaces for which we will prove infinite-dimensional Ramsey theorems in Section \ref{sec.MainTheorem}.

Each diagonal coding antichain $M\in\mathcal{D}$  can also be identified  with the tree induced by its meet-closure.
The set of coding and splitting nodes in (the tree induced by) $M$ are called the {\em critical nodes} of $M$, and we  let 
 $\lgl d^M_n:n\in \mathbb{N}\rgl$ denote their 
enumeration  in order of increasing length.
For $n \in \mathbb{N}$, $M(n)$ denotes the set of nodes in $M$ of length $|d^M_n|$.
Given  $k\in \mathbb{N}$, $r_k(M)$ denotes  the finite subset of $M$ consisting of all nodes in $M$ with length less than  $|d^M_{k}|$.
Thus,  $r_0(M)$ is the empty set
and 
\begin{equation}
r_k(M)=\bigcup_{n<k}M(n).
\end{equation}
We define the following  notation in  line with topological Ramsey space theory  from
\cite{TodorcevicBK10}.
For $k\in\mathbb{N}$, define
\begin{equation}
\mathcal{AD}_k=\{r_k(M):M\in\mathcal{D}\},
\end{equation}
the set of all $k$-th restrictions of members of $\mathcal{D}$.
Let
\begin{equation}
\mathcal{AD}=\bigcup_{k=1}^{\infty}\mathcal{AD}_k,
\end{equation}
the set of all finite approximations to members of $\mathcal{D}$.
Note that having identified $M$ with the tree it induces, members of $\mathcal{AD}$ are finite diagonal trees which are initial segments of the diagonal tree $M$.

For $A,B\in\mathcal{AD}$
we write
 $A\sqsubseteq B$  if and only if
 there is some $M\in\mathcal{D}$ and some $j\le k$ such that
 $A=r_j(M)$ and $B=r_k(M)$.
 In this case,
  $A$ is called  an  {\em initial segment} of $B$; we also
say that $B$ {\em extends} $A$.
If $A\sqsubseteq B$ and $A\ne B$, then we say that $A$ is a {\em proper initial segment} of $B$ and write $A\sqsubset B$.
When $A=r_j(M)$  for some $j$,
we  also write $A\sqsubset M$
 and call $A$ an
 {\em initial segment} of $M$.

The {\em metric topology} on $\mathcal{D}$ is the topology induced by basic open cones   of the form
\begin{equation}
[A,\bD]=\{M\in \mathcal{D}: \exists k\,  (r_k(M)=  A)\},
\end{equation}
 for  $A\in\mathcal{AD}$.
The {\em Ellentuck topology} on $\mathcal{D}$ is induced by basic open sets of the form
\begin{equation}
[A,M]=\{N\in \mathcal{D}:\exists k\,(r_k(N)=A)\mathrm{\ and\ } N\le M\},
\end{equation}
where $A\in\mathcal{AD}$ and $M\in\mathcal{D}$.
Thus, the Ellentuck topology refines the metric topology.

Given $A\in\mathcal{AD}$,
let $\ell_A$ denote the maximum of the lengths of nodes in $A$,
and
 let 
\begin{equation}
\max(A)=\{s\in A:|s|=\ell_A\}.
\end{equation}
The partial ordering $\le_{\mathrm{fin}}$ on $\mathcal{AD}$ is defined as follows:
For   $A,B\in\mathcal{AD}$,
write
$A\le_{\mathrm{fin}} B$ if and only if  $A$ is a subtree of $B$.
 Define  $\depth_M(A)$ to be   the least  integer $k$ such that  $A\le_{\mathrm{fin}}r_k(M)$, if it exists; otherwise, define $\depth_M(A)=\infty$.
Lastly, given  $j<k$, $A\in\mathcal{AD}_j$ and $M\in\mathcal{D}$, define
\begin{equation}
r_k[A,M]=\{r_k(N):N\in [A,M]\}.
\end{equation}
\end{defn}




\section{Forcing the Extended  Pigeonhole Principle}\label{sec.APP}

This section proves an enhanced pigeonhole principle for Baire spaces of diagonal coding antichains which preserves the  width in  some finite initial segment of the ambient antichain. 
 This is done to prove the pigeonhole principle (axiom {\bf A.4} of Todorcevic) while simultaneously  overcoming the fact 
that the amalgamation axiom {\bf A.3}(2)  of Todorcevic does not hold for many spaces of the form $\mathcal{D}(\bD)$.
The proof will use 
 forcing techniques to do infinitely many unbounded searches for  finite objects with some homogeneity properties.  
Since the objects are finite, they exist in the ground model; no generic extension is needed for the main theorem of this section.

\begin{defn}[Good Diagonal Coding Antichains]\label{defn.CDABetter}
Fix an enumerated
\Fraisse\ structure  $\bK$ satisfying \SDAP$^+$.
We call a  diagonal  coding antichain $M$  {\em good} if it satisfies the following:
\begin{enumerate}
\item
For each $n\in\mathbb{N}$,
the longest   splitting node  in $M$ with length less than 
$|c^{M}_n|$ 
extends $\prec$-right to $c^{M}_n$.
We call this splitting node the {\em splitting predecessor} of $c^{M}_n$ and denote it by 
$\mathrm{sp}_M(c^{M}_n)$.
\item
For any   $m<n$,
letting $s$ be the $\prec$-left extension of 
$\mathrm{sp}_M(c^{M}_m)$ 
in $M\re (|c^{M}_m|+1)$ and 
$t$ be the $\prec$-left extension of 
$\mathrm{sp}_M(c^{M}_n)$
in $M\re (|c^{M}_n|+1)$,
we have  $s(c^{M}_m)\sim  t(c^{M}_n)$.
\item
There is a $k\in \mathbb{N}$ such that 
for all $n\ge k$,
to  each $1$-type $\sigma$ over $\bK_{n+1}$
there corresponds a unique node $s\in M \re (|c^{M}_n|+1)$ 
such that $\type(s/\bM_{n+1})\sim \sigma$.
\end{enumerate}
\end{defn}

Fix throughout this section  a good diagonal coding antichain  $\bD$ for $\bK$, and let $\mathcal{D}$ denote $\mathcal{D}(\bD)$.
For any subset $U\sse \bD$, finite or infinite,
 let
$L_U$ denote $\{|t|:t\in U\}$, the set of lengths of nodes in $U$,
and let 
$U^{\wedge}$ denote the meet-closure of $U$.
Define 
\begin{equation}
\widehat{U}=\{t\re \ell:t\in U\mathrm{\ and\ }\ell\le|t|\},
\end{equation}
the tree of all initial segments  of members of $U$,
and 
\begin{equation}
\tree({U})=\{t\in \widehat{U}:
|t|\in L_{U^{\wedge}}\},
\end{equation}
the tree induced by the meet-closure of $U$.
We will abuse notation an identify $U$ with tree$(U)$.

Let
\begin{equation}\label{eq.AThat}
\widehat{\mathcal{AD}}=\{A\rl \ell:A\in\mathcal{AD}\mathrm{\ and\ }\ell\le \ell_A\}.
\end{equation}
Given $M\in \mathcal{D}$,
let  $\mathcal{AD}(M)$ denote the members of $\mathcal{AD}$ which are contained in $M$.
For $k\in\mathbb{N}$, let $\mathcal{AD}_k(M)$ denote the set of those $A\in\mathcal{AD}_k$ such that $A$ is a subtree of $M$.
 Define
\begin{equation}\label{eq.AThatT}
\widehat{\mathcal{AD}}(M)=\{A\rl \ell:
A\in  \mathcal{AD}(M) \mathrm{\ and\ } \ell\in L_M\}.
\end{equation}
Note that for any $M\in\mathcal{D}$,  there are  members of $\widehat{\mathcal{AD}}(M)$ which are not similar to  $r_n(\bD)$ for any $n$, and hence are not members of $\mathcal{AD}(M)$.

\begin{defn}\label{defnhatAT}
Given $M\in \mathcal{D}$ and
$B\in\widehat{\mathcal{AD}}(M)$,
letting $m$ be the least integer
for which there exists $B'\in\mathcal{AD}_m$ such that $\max(B)\sqsubseteq \max(B')$,
define
\begin{equation}\label{eq.AT}
[B,M]^*=\{N\in\mathcal{D}:
\max(B)
\sqsubseteq \max(r_m(N))  \mathrm{\ and\ }  N\le M\}.
\end{equation}
For $n\ge m$, define
\begin{equation}\label{eq.rn*}
r_n[B,M]^*=\{r_n(N):N\in [B,M]^*\},
\end{equation}
and let
\begin{equation}
r[B,M]^*=\bigcup_{m\le n} r_n[B,M]^*.
\end{equation}
\end{defn}

Given $B\in\widehat{\mathcal{AD}}$ and  $M\in\mathcal{D}$,
notice that the set 
$[B,M]^*$ 
from Definition \ref{defnhatAT}
 is open in the Ellentuck topology on $\mathcal{D}$:
If $B$ is in $\mathcal{AD}_k$ for some $k$,
then  the set  $[B,M]^*$ is the union of
$[B,M]$ along with
 all $[C,M]$, where
$C \in\mathcal{AD}_k$ and $\max(C)$ end-extends $\max(B)$.
If $B$ is in $\widehat{\mathcal{AD}}$ but not in $\mathcal{AD}$,
then
 letting  $k$ be the least integer for which there is some $C\in\mathcal{AD}_k$  with $\max(C)\sqsupset \max(B)$,
 we see that
$[B,M]^*$ equals the union of  all $[C,M]$, where $C\in\mathcal{AD}_{k}$ and $B \sqsubseteq C$.
For the same reasons, the set $[B, \bD]^*$ is open in the  metric topology on $\mathcal{D}$.
Notice also  that
$r_{n}[B, M]^*$ defined in equation (\ref{eq.rn*})
is equal to
$\{C\in\mathcal{AD}_{n}(M):
\max(B)\sqsubseteq \max(C)\}$.

Given $M\in\mathcal{D}$,
 a splitting node $s\in M$ is  called a
{\em splitting predecessor  of a coding node in $M$}
(or just {\em splitting predecessor} if $M$ is understood)
 if and only if
 there is a coding node $c\in M$
 such that $s\subset c$ and
 $|c|$ is minimal in
$L_M$ above  $|s|$.
Given a coding node $c$ in $M$, we write
$\splitpred_M(c)$ to denote the splitting predecessor of $c$ in $M$.
Note that $s=\splitpred_M(c)$ if and only if the minimal node  in $M$ extending the $\prec$-right  extension of $s$ is a coding node.
When $M=\bD$, we will  usually write $\splitpred(c)$ in place  of 
$\splitpred_{\bD}(c)$.

Given   $A\in \widehat{\mathcal{AD}}(M)$,
let  $A^+$ denote the union of $A$  with the set of   immediate successors  in    $\widehat{M}$ of the members of $\max(A)$; thus,
\begin{equation}
A^+=A\cup \{t\in M\re (\ell_{A}+1):t\re \ell_A\in A\}
\end{equation}
and    $\max(A^+)$ is the   set of all
nodes  in $A^+$ with  length $\ell_A+1$.

A set of nodes is called a {\em level set} if all nodes in the set have the same length. 
For  level sets $X,Y\sse\bD$, we say that $Y$ {\em end-extends} $X$ and 
write $X\sqsubset Y$ if and only if $X$ and $Y$ have the same cardinality, $\ell_X<\ell_Y$,  and $Y\re \ell_X=X$.
More generally, for $A,B\in\widehat{\mathcal{AD}}$, write 
$A\sqsubseteq B$ 
if and only if  $A= B\rl \ell_A$;
in this case, 
write $A\sqsubset B$ if  also $\ell_A<\ell_B$.
\vskip.1in

\begin{assumption}\label{assump.matrixHL}
Fix a good diagonal coding antichain $\bD$ and  let $\mathcal{D}:=\mathcal{D}(\bD)$.
We will be working with  triples $(A,B,k)$,
where
$A\in\widehat{\mathcal{AD}}$,
$B\sse\widehat{\bD}$, and 
$A\sqsubset B\sse  A^+$.
Assume that 
{\bf all splitting nodes in $A$,  $B$, and $D$  are not splitting predecessors in $\bD$}.
We consider the following combinations of one of Cases (a) or (b) with one of Cases (i) or  (ii).
\vskip.1in

\begin{enumerate}
\item[]
\begin{enumerate}
\item[\bf Case (a).]
$\max(r_{k+1}(\bD))$ has a splitting node.
\end{enumerate}
\end{enumerate}

\begin{enumerate}
\item[]
\begin{enumerate}
\item[\bf Case (b).]
$\max(r_{k+1}(\bD))$
has a coding node.
\end{enumerate}
\end{enumerate}

\begin{enumerate}
\item[]
\begin{enumerate}
\item[\bf Case (i).]
$k\ge 1$,
$A\in \mathcal{AD}_k$,
and
 $B=A^+$.
 \end{enumerate}
\end{enumerate}

\begin{enumerate}
\item[]
\begin{enumerate}
\item[\bf Case (ii).]
$k\ge 0$, $A$ has at least one node,
each member of $\max(A)$ has exactly one extension in $\max(B)$ (that is, $\max(A)\sqsubset \max(B)$), and
$A = C\rl \ell$ for some $C\in\mathcal{AD}_{k+1}$ 
 and  $\ell<\ell_C$
such that
$r_{k}(C)\sqsubseteq A$ and 
$B\sqsubseteq C$.
\end{enumerate}
\end{enumerate}
We point out that 
in Case (ii), $A$ may or may not be a member of
$\mathcal{AD}$.
\end{assumption}

The following theorem of \Erdos\ and Rado  will
be used in the proof of Theorem \ref{thm.matrixHL}.

\begin{thm}[\Erdos-Rado]\label{thm.ER}
For $r<\om$ and $\mu$ an infinite cardinal,
$$
\beth_r(\mu)^+\ra(\mu^+)_{\mu}^{r+1}.
$$
\end{thm}

We are now set up to prove the theorem which will form  the basis of the result   that all Borel sets in $\mathcal{B}(\bD)$ are Ramsey.

\begin{thm}[Extended Pigeonhole Principle]\label{thm.matrixHL}
Let $\bD$,   $(A,B,k)$, $D$, $d$ be as in Assumption 
\ref{assump.matrixHL},
where  $(A, B,k)$ satisfies one of Cases (i) or (ii) and  one of Cases (a) or (b).
Let
  $h:  r_{k+1}[D,\bD]^*\ra 2$  be a coloring.
Then  there is an
$N\in [D,\bD]^*$
such that
$h$ is monochromatic on $ r_{k+1}[B,N]^*$.
\end{thm}

\begin{proof}
Assume the hypotheses.
Let $\mathbf{i}+1$ be the number of  nodes in  $\max(B)$,
and  fix an enumeration   $s_0,\dots, s_{\mathbf{i}}$  of
 the nodes in    $\max(B)$
 with the property that  for any
$C\in r_{k+1}[B,\bD]^*$,
the critical  node in $\max(C)$ extends $s_{\mathbf{i}}$.
Note that in Case (b), $\mathbf{i}$ must be at least two. 
Let $d$ denote the integer such that $D\in\mathcal{AD}_d$, and let
$I$ denote the set  of all  
$n>d$  
such that for some (equivalently, all) $M\in [D,\bD]$
there is 
 a member
$C\in r_{k+1}[B,M]^*$ 
with $\depth_M(C)=n$.
Let $L$ denote
the set 
$\{\ell_{r_n(M)}: M\in [D,\bD]$ and $n\in I\}$.
In  Case (b),
for  $\ell\in L$ we 
let $\ell'$ denote the length of the splitting predecessor in $\bD$ of the coding node  in $\bD$ of length $\ell$, and let $L'=\{\ell':\ell\in L\}$.

Given $U\in \mathcal{AD}\cup \mathcal{D}$ with $D\sqsubseteq U$, define  the set $\Ext_U(B)$ as follows:
In Case (a), let 
$\Ext_U(B)$
consist of those  level
sets $X\sse U$  such that 
$X=\max(C)$  for some  $C\in r_{k+1}[B,\bD]^*$,  
where the splitting node in $X$ is  not a  splitting predecessor in $\bD$.
In Case (b),
let 
$\Ext_U(B)$
consist of those 
sets $X\sse U$ 
such that 
$X$ consists of the non-coding  nodes in
 $\max(C)$
along with the splitting predecessor in $\bD$ of the coding node in $\max(C)$,
 for some  $C\in r_{k+1}[B,\bD]^*$.
We will simply  write $\Ext(B)$  to mean $\Ext_{\bD}(B)$.

The coloring $h$ induces a coloring $h': \Ext(B)\ra 2$ as follows:
For  $X\in  \Ext(B)$,
in Case (a)
define
$h'(X)=h(C)$, where
$C$ is the member of $r_{k+1}[B,\bD]^*$
such that $X=\max(C)$.
In Case (b),
define $h'(X)=h(C)$, where
$C$ is the member of  $r_{k+1}[B,\bD]^*$
such that
$X=(\max(C)\setminus\{c\})\cup\{\splitpred(c)\}$, where $c$ denotes the coding node in $\max(C)$.

For
  $i\le \mathbf{i}$,   let  $T_i=\{t\in \widehat{\bD}:t\contains s_i\}$.
Let $\kappa=\beth_{2\mathbf{i}}$, so that the partition relation $\kappa\ra (\aleph_1)^{2\mathbf{i}}_{\aleph_0}$ holds by the \Erdos-Rado Theorem \ref{thm.ER}.
The following forcing notion   adds $\kappa$ many paths through  each  $T_i$,  $i< \mathbf{i}$,
and one path  through $T_{\mathbf{i}}$.

In  both Cases (a) and (b),
define
$\bP$ to be  the set of finite partial functions $p$   such that
\begin{enumerate}
\item
$\dom(p)=\{\mathbf{i}\}\times \vec{\delta}_p$, where $\vec{\delta}_p$ is a finite subset of $\kappa$;
 \item
 $p(\mathbf{i})\in T_{\mathbf{i}}$ and 
 $\{p(i,\delta) : \delta\in  \vec{\delta}_p\}\sse  T_i$ for each $i<\mathbf{i}$;
\item
 For  any  choices of $\delta_i\in\vec{\delta}_p$,  $i<\mathbf{i}$,
 the  set $\{p(i,\delta_i):i<\mathbf{i}\}\cup\{p(\mathbf{i})\}$ is a member of $\Ext(B)$.
 \end{enumerate}
Let $\ell_p$ denote the maximal length of nodes in $\ran(p)$.
All nodes in $\ran(p)$ will have length $\ell_p$
except
in Case (b),
where 
 the splitting predecessor $p(\mathbf{i})$    will have length $\ell_p'$.

The partial
 ordering on $\bP$  is just reverse inclusion:
$q\le p$ if and only if  $\ell_q\ge \ell_p$,  $\vec{\delta}_q\contains \vec{\delta}_p$,
$q(\mathbf{i})\contains p(\mathbf{i})$, and
$q(i,\delta)\contains p(i,\delta)$ for each $(i,\delta)\in \mathbf{i}\times \vec{\delta}_p$.
It is routine to check that    $(\bP,\le)$ is a separative, atomless partial order.

We point out that condition (3) in the definition of $\bP$  is easy to satisfy since 
$\bK$ has SDAP$^+$:
Let  $C$ be any member of $ r_{k+1}[B,\bD]^*$  and let 
$\lgl t_i:i\le \mathbf{i}\rgl$ enumerate $\max(C)$ so that
so that each $t_i$ extends $s_i$.
In Case (a),
each $p(i,\delta)$ only need be a node in $T_i$ of length $\ell_p$.
If   (1) of the Extension Property holds for 
$\bK$, 
then  $p(\mathbf{i})$ just needs to be a splitting node in $T_{\mathbf{i}}$ which is not a splitting predecessor;
if (2) of the Extension Property holds,
it suffices for $p(\mathbf{i})$ to additionally satisfy $\psi(p(\mathbf{i}))=\psi(t_{\mathbf{i}})$.
In Case (b),
$p(\mathbf{i})$ need only be the splitting predecessor of some  coding node $c$  in $T_{\mathbf{i}}$, and 
each $p(i,\delta)$ need only be a node in $T_i$ of length $\ell_p$ such that $p(i,\delta)^+(c)\sim t_i^+(t_{\mathbf{i}})$.

Given $p\in\bP$,
 the  range of $p$ is  the set
$$
\ran(p)=\{p(i,\delta):(i,\delta)\in \mathbf{i}\times \vec\delta_p\}\cup \{p(\mathbf{i})\}.
$$
If also $q\in \bP$ and $q\le p$, then
 we let 
\begin{equation}
\ran(q)\re \dom (p)=\{q(i,\delta):(i,\delta)\in \mathbf{i}\times \vec{\delta}_p\}\cup \{q(\mathbf{i})\}.
\end{equation}
For  $(i,\al)\in \mathbf{i}\times \kappa$,
 let
\begin{equation}
 \dot{b}_{i,\al}=\{\lgl p(i,\al),p\rgl:p\in \bP\mathrm{\  and \ }\al\in\vec{\delta}_p\},
\end{equation}
 a $\bP$-name for the $\al$-th generic branch through $T_i$.
Let
\begin{equation}
\dot{b}_{\mathbf{i}}=\{\lgl p(\mathbf{i}),p\rgl:p\in\bP\},
\end{equation}
a $\bP$-name for the generic branch through $M_{\mathbf{i}}$.
 Given a generic filter   $G\sse \bP$, notice that
 $\dot{b}_{\mathbf{i}}^G=\{p(\mathbf{i}):p\in G\}$,
 which  is a cofinal path   in $T_{\mathbf{i}}$.
We point out that given  $p\in \bP$,
\begin{equation}
 p\forces  \forall
 (i,\al)\in \mathbf{i}\times \vec\delta_p\, (\dot{b}_{i,\al}\re \ell_p= p(i,\al)).
\end{equation}
Furthermore, in Case (a),
 $p\forces (\dot{b}_{\mathbf{i}}\re \ell_p=p(\mathbf{i}))$,
  while in Case (b),
 $p\forces (\dot{b}_{\mathbf{i}}\re
  \ell'_p=p(\mathbf{i}))$.

 Let $\dot{G}$ be the $\bP$-name for a generic filter, and let $\dot{L}_G$ be a $\bP$-name for the set of lengths  $\{\ell_p: p\in\dot{G}\}$.
 Note that
 $\bP$ forces that $\dot{L}_G\sse L$.
Let $\dot{\mathcal{U}}$ be a $\bP$-name for a non-principal ultrafilter on $\dot{L}_G$.

We will write sets $\{\al_i:i< \mathbf{i}\}$ in $[\kappa]^{\mathbf{i}}$ as vectors $\vec{\al}=\lgl \al_0,\dots,\al_{\mathbf{i}-1}\rgl$ in strictly increasing order.
For $\vec{\al}\in[\kappa]^{\mathbf{i}}$,  let
\begin{equation}
\dot{b}_{\vec{\al}}=
\lgl \dot{b}_{0,\al_0},\dots, \dot{b}_{\mathbf{i}-1,\al_{\mathbf{i}-1}},\dot{b}_{\mathbf{i}}\rgl.
\end{equation}
For  $\ell\in L$,
in Case (a)  let
 \begin{equation}
 \dot{b}_{\vec\al}\re \ell=
 \lgl \dot{b}_{0,\al_0}\re \ell,\dots, \dot{b}_{\mathbf{i}-1,\al_{\mathbf{i}-1}}\re \ell,\dot{b}_{\mathbf{i}}\re \ell\rgl;
 \end{equation}
 and in Case (b),
 let
  \begin{equation}
 \dot{b}_{\vec\al}\re \ell=
 \lgl \dot{b}_{0,\al_0}\re \ell,\dots, \dot{b}_{\mathbf{i}-1,\al_{\mathbf{i}-1}}\re \ell,\dot{b}_{\mathbf{i}}\re \ell'\rgl.
 \end{equation}
 Note that
$h'$  is a coloring on
 $\dot{b}_{\vec\al}\re \ell$
whenever this is
 forced to be a member of $\Ext_M(B)$.
Given $\vec{\al}\in [\kappa]^{\mathbf{i}}$ and  $p\in \bP$ with $\vec\al\sse\vec{\delta}_p$,
let
\begin{equation}
X(p,\vec{\al})=\{p(i,\al_i):i<\mathbf{i}\}\cup\{p(\mathbf{i})\}.
\end{equation}


For each $\vec\al\in[\kappa]^{\mathbf{i}}$,
choose a condition $p_{\vec{\al}}\in\bP$ satisfying the following:
\begin{enumerate}
\item
 $\vec{\al}\sse\vec{\delta}_{p_{\vec\al}}$.
\item
There is an $\varepsilon_{\vec{\al}}\in 2$
 such that
$p_{\vec{\al}}\forces$
``$h'(\dot{b}_{\vec{\al}}\re \ell)=\varepsilon_{\vec{\al}}$
for $\dot{\mathcal{U}}$ many $\ell$ in $\dot{L}_G$''.
\item
$h'(X(p_{\vec\al},\vec{\al}))=\varepsilon_{\vec{\al}}$.
\end{enumerate}

Such conditions can be found as follows:
Fix some $X\in \Ext(B)$ and let $x_i$ denote the node in $X$ extending $s_i$, for each $i\le \mathbf{i}$.
For $\vec{\al}\in[\kappa]^{\mathbf{i}}$,  define
$$
p^0_{\vec{\al}}=\{\lgl (i,\delta), x_i\rgl: i< \mathbf{i}, \ \delta\in\vec{\al} \}\cup\{\lgl \mathbf{i},x_{\mathbf{i}}\rgl\}.
$$
Then
 (1) will  hold for all $p\le p^0_{\vec{\al}}$,
 since $\vec\delta_{p_{\vec\al}^0}= \vec\al$.
 Next,
let  $p^1_{\vec{\al}}$ be a condition below  $p^0_{\vec{\al}}$ which
forces  $h'(\dot{b}_{\vec{\al}}\re \ell)$ to be the same value for
$\dot{\mathcal{U}}$  many  $\ell\in \dot{L}_G$.
Extend this to some condition
 $p^2_{\vec{\al}}\le p_{\vec{\al}}^1$
 which
 decides a value $\varepsilon_{\vec{\al}}\in 2$
 so that
 $p^2_{\vec{\al}}$ forces
  $h'(\dot{b}_{\vec{\al}}\re \ell)=\varepsilon_{\vec{\al}}$
for $\dot{\mathcal{U}}$ many $\ell$ in $\dot{L}_G$.
Then
 (2) holds for all $p\le p_{\vec\al}^2$.
If $ p_{\vec\al}^2$ satisfies (3), then let $p_{\vec\al}=p_{\vec\al}^2$.
Otherwise,
take  some  $p^3_{\vec\al}\le p^2_{\vec\al}$  which forces
$\dot{b}_{\vec\al}\re \ell\in \Ext(B)$ and
$h'(\dot{b}_{\vec\al}\re \ell)=\varepsilon_{\vec\al}$
for
some $\ell\in\dot{L}_G$
 with
$\ell_{p^2_{\vec\al}}< \ell\le \ell_{p^3_{\vec\al}}$;
then $h'(X(p^3_{\vec\al}\re \ell,\vec\al))
=\varepsilon_{\vec\al}$.
Thus, letting
$p_{\vec\al}$ be $p^3_{\vec\al}\re \ell$, we see that
 $p_{\vec\al}$ satisfies (1)--(3).

Let $\mathcal{I}$ denote the collection of all functions $\iota: 2\mathbf{i}\ra 2\mathbf{i}$ such that
for each $i<\mathbf{i}$,
$\{\iota(2i),\iota(2i+1)\}\sse \{2i,2i+1\}$.
For $\vec{\theta}=\lgl \theta_0,\dots,\theta_{2\mathbf{i}-1}\rgl\in[\kappa]^{2\mathbf{i}}$,
$\iota(\vec{\theta}\,)$ determines the pair of sequences of ordinals $\lgl \iota_e(\vec{\theta}\,),\iota_o(\vec{\theta}\,)\rgl$, where
\begin{align}
\iota_e(\vec{\theta}\,)&=
\lgl \theta_{\iota(0)},\theta_{\iota(2)},\dots,\theta_{\iota(2\mathbf{i}-2))}\rgl\cr
\iota_o(\vec{\theta}\,)&=
 \lgl\theta_{\iota(1)},\theta_{\iota(3)},\dots,\theta_{\iota(2\mathbf{i}-1)}\rgl.
 \end{align}

We now proceed to  define a coloring  $f$ on
$[\kappa]^{2\mathbf{i}}$ into countably many colors.
Let $\vec{\delta}_{\vec\al}$ denote $\vec\delta_{p_{\vec\al}}$,
 $k_{\vec{\al}}$ denote $|\vec{\delta}_{\vec\al}|$,
$\ell_{\vec{\al}}$ denote  $\ell_{p_{\vec\al}}$, and let $\lgl \delta_{\vec{\al}}(j):j<k_{\vec{\al}}\rgl$
denote the enumeration of $\vec{\delta}_{\vec\al}$
in increasing order.
Given  $\vec\theta\in[\kappa]^{2\mathbf{i}}$ and
 $\iota\in\mathcal{I}$,   to reduce subscripts
 let
$\vec\al$ denote $\iota_e(\vec\theta\,)$ and $\vec\beta$ denote $\iota_o(\vec\theta\,)$, and
define
\begin{align}\label{eq.fiotatheta}
f(\iota,\vec\theta\,)= \,
&\lgl \iota, \varepsilon_{\vec{\al}}, k_{\vec{\al}}, p_{\vec{\al}}(\mathbf{i}),
\lgl \lgl p_{\vec{\al}}(i,\delta_{\vec{\al}}(j)):j<k_{\vec{\al}}\rgl:i< \mathbf{i}\rgl,\cr
& \lgl  \lgl i,j \rgl: i< \mathbf{i},\ j<k_{\vec{\al}},\ \mathrm{and\ } \delta_{\vec{\al}}(j)=\al_i \rgl,\cr
&\lgl \lgl j,k\rgl:j<k_{\vec{\al}},\ k<k_{\vec{\beta}},\ \delta_{\vec{\al}}(j)=\delta_{\vec{\beta}}(k)\rgl\rgl.
\end{align}
Fix some ordering of $\mathcal{I}$ and define
\begin{equation}
f(\vec{\theta}\,)=\lgl f(\iota,\vec\theta\,):\iota\in\mathcal{I}\rgl.
\end{equation}

By the \Erdos-Rado Theorem  \ref{thm.ER},  there is a subset $K\sse\kappa$ of cardinality $\aleph_1$
which is homogeneous for $f$.
Take $K'\sse K$ so that between each two members of $K'$ there is a member of $K$.
Then take  $K_i\sse K'$ satisfying
  $K_0<\dots<K_{\mathbf{i}-1}$,
where $K_i<K_{i+1}$ means that each ordinal in $K_i$ is less than each ordinal in $K_{i+1}$.
Let $\vec{K}$ denote $\prod_{i<\mathbf{i}}K_i$.

Fix some  $\vec\gamma\in \vec{K}$, and define
\begin{align}\label{eq.star}
&\varepsilon_*=\varepsilon_{\vec\gamma},\ \
k_*=k_{\vec\gamma},\ \
t_{\mathbf{i}}=p_{\vec\gamma}(\mathbf{i}),\cr
t_{i,j}&=p_{\vec{\gamma}}(i,\delta_{\vec{\gamma}}(j))\mathrm{\ for\ }
i<\mathbf{i},\ j<k_*.
\end{align}
The next three lemmas show that the
 values  in equation (\ref{eq.star}) are the same for any choice of
$\vec\gamma$ in $\vec{K}$.

\begin{lem}\label{lem.onetypes}
 For all $\vec{\al}\in \vec{K}$,
 $\varepsilon_{\vec{\al}}=\varepsilon_*$,
$k_{\vec\al}=k_*$,  $p_{\vec{\al}}(\mathbf{i})=t_\mathbf{i}$, and
$\lgl p_{\vec\al}(i,\delta_{\vec\al}(j)):j<k_{\vec\al}\rgl
=
 \lgl t_{i,j}: j<k_*\rgl$ for each $i< \mathbf{i}$.
\end{lem}

\begin{proof}
Let
 $\vec{\al}$ be any member of $\vec{K}$, and let $\vec{\gamma}$ be the set of ordinals fixed above.
Take  $\iota\in \mathcal{I}$
to be the identity function on $2\mathbf{i}$.
Then
there are $\vec\theta,\vec\theta'\in [K]^{2\mathbf{i}}$
such that
$\vec\al=\iota_e(\vec\theta\,)$ and $\vec\gamma=\iota_e(\vec\theta'\,)$.
Since $f(\iota,\vec\theta\,)=f(\iota,\vec\theta'\,)$,
it follows that $\varepsilon_{\vec\al}=\varepsilon_{\vec\gamma}$, $k_{\vec{\al}}=k_{\vec{\gamma}}$, $p_{\vec{\al}}(\mathbf{i})=p_{\vec{\gamma}}(\mathbf{i})$,
and $\lgl \lgl p_{\vec{\al}}(i,\delta_{\vec{\al}}(j)):j<k_{\vec{\al}}\rgl:i< \mathbf{i}\rgl
=
\lgl \lgl p_{\vec{\gamma}}(i,\delta_{\vec{\gamma}}(j)):j<k_{\vec{\gamma}}\rgl:i< \mathbf{i}\rgl$.
\end{proof}

Let $\ell_*$ denote the length of the  nodes
 $t_{i,j}$, $(i,j)\in d\times k_*$.
 In Case (a), $|t_{\mathbf{i}}|$ also equals $\ell_*$; in Case (b), let $\ell_*'$ denote 
$|t_{\mathbf{i}}|$.

\begin{lem}\label{lem.j=j'}
Given any $\vec\al,\vec\beta\in \vec{K}$,
if $j,k<k_*$ and $\delta_{\vec\al}(j)=\delta_{\vec\beta}(k)$,
 then $j=k$.
\end{lem}

\begin{proof}
Let $\vec\al,\vec\beta$ be members of $\vec{K}$   and suppose that
 $\delta_{\vec\al}(j)=\delta_{\vec\beta}(k)$ for some $j,k<k_*$.
For  $i<\mathbf{i}$, let  $\rho_i$ be the relation from among $\{<,=,>\}$ such that
 $\al_i\,\rho_i\,\beta_i$.
Let   $\iota$ be the member of  $\mathcal{I}$  such that for each $\vec\theta\in[K]^{2\mathbf{i}}$ and each $i<\mathbf{i}$,
$\theta_{\iota(2i)}\ \rho_i \ \theta_{\iota(2i+1)}$.
Fix some
$\vec\theta\in[K']^{2\mathbf{i}}$ such that
$\iota_e(\vec\theta)=\vec\al$ and $\iota_o(\vec\theta)= \vec\beta$.
Since between any two members of $K'$ there is a member of $K$, there is a
 $\vec\zeta\in[K]^{\mathbf{i}}$ such that  for each $i< \mathbf{i}$,
 $\al_i\,\rho_i\,\zeta_i$ and $\zeta_i\,\rho_i\, \beta_i$.
Let   $\vec\mu,\vec\nu$ be members of $[K]^{2\mathbf{i}}$ such that $\iota_e(\vec\mu)=\vec\al$,
$\iota_o(\vec\mu)=\vec\zeta$,
$\iota_e(\vec\nu)=\vec\zeta$, and $\iota_o(\vec\nu)=\vec\beta$.
Since $\delta_{\vec\al}(j)=\delta_{\vec\beta}(k)$,
the pair $\lgl j,k\rgl$ is in the last sequence in  $f(\iota,\vec\theta)$.
Since $f(\iota,\vec\mu)=f(\iota,\vec\nu)=f(\iota,\vec\theta)$,
also $\lgl j,k\rgl$ is in the last  sequence in  $f(\iota,\vec\mu)$ and $f(\iota,\vec\nu)$.
It follows that $\delta_{\vec\al}(j)=\delta_{\vec\zeta}(k)$ and $\delta_{\vec\zeta}(j)=\delta_{\vec\beta}(k)$.
Hence, $\delta_{\vec\zeta}(j)=\delta_{\vec\zeta}(k)$,
and therefore $j$ must equal $k$.
\end{proof}

For each $\vec\al\in \vec{K}$, given any
$\iota\in\mathcal{I}$, there is a $\vec\theta\in[K]^{2\mathbf{i}}$ such that $\vec\al=\iota_o(\vec\al)$.
By the second line of equation  (\ref{eq.fiotatheta}),
there is a strictly increasing sequence
$\lgl j_i:i< \mathbf{i}\rgl$  of members of $k_*$ such that
$\delta_{\vec\gamma}(j_i)=\al_i$.
By
homogeneity of $f$,
this sequence $\lgl j_i:i< \mathbf{i}\rgl$  is the same for all members of $\vec{K}$.
Then letting
 $t^*_i$ denote $t_{i,j_i}$,
 one sees that
\begin{equation}
p_{\vec\al}(i,\al_i)=p_{\vec{\al}}(i, \delta_{\vec\al}(j_i))=t_{i,j_i}=t^*_i.
\end{equation}
Let $t_{\mathbf{i}}^*$ denote $t_{\mathbf{i}}$.

\begin{lem}[Homogeneity of $\{p_{\vec{\al}}:\vec{\al}\in \vec{K}\,\}$]\label{lem.compat}
For any finite subset $\vec{J}\sse \vec{K}$,
$p_{\vec{J}}:=\bigcup\{p_{\vec{\al}}:\vec{\al}\in \vec{J}\,\}$
is a member of $\bP$ which is below each
$p_{\vec{\al}}$, $\vec\al\in\vec{J}$.
\end{lem}

\begin{proof}
Given  $\vec\al,\vec\beta\in \vec{J}$,
if
 $j,k<k_*$ and
 $\delta_{\vec\al}(j)=\delta_{\vec\beta}(k)$, then
 $j$ and $k$ must be equal, by
 Lemma  \ref{lem.j=j'}.
Then  Lemma \ref{lem.onetypes} implies
that for each $i<\mathbf{i}$,
\begin{equation}
p_{\vec\al}(i,\delta_{\vec\al}(j))=t_{i,j}=p_{\vec\beta}(i,\delta_{\vec\beta}(j))
=p_{\vec\beta}(i,\delta_{\vec\beta}(k)).
\end{equation}
Hence,
 for all
$\delta\in\vec{\delta}_{\vec\al}\cap
\vec{\delta}_{\vec\beta}$
and  $i<\mathbf{i}$,
$p_{\vec\al}(i,\delta)=p_{\vec\beta}(i,\delta)$.
Thus,
$p_{\vec{J}}:=
\bigcup \{p_{\vec{\al}}:\vec\al\in\vec{J}\}$
is a  function with $\vec{\delta}_{p_{\vec{J}}}=
\bigcup\{
\vec{\delta}_{\vec\al}:
\vec\al\in\vec{J}\,\}$; hence,
$p_{\vec{J}}$ is a member of $\bP$.
Since
for each $\vec\al\in\vec{J}$,
$\ran(p_{\vec{J}}\re \vec{\delta}_{\vec\al})=\ran(p_{\vec\al})$,
it follows that
$p_{\vec{J}}\le p_{\vec\al}$ for each $\vec\al\in\vec{J}$.
\end{proof}


We now proceed to  build
an
 $N\in [D,\bD]$
so that  the coloring
$h'$ will be  monochromatic on $\Ext_N(B)$,
from which it will follow that $h$ is  monochromatic in $r_{k+1}[B,N]^*$.
 Set
\begin{equation}\label{eq.U*}
U^*=\{t^*_i:i\le \mathbf{i}\}\cup\{u^*:u\in \max(D)^+\setminus B\},
\end{equation}
where  for each $u$
 in $\max(D)^+\setminus B$,
$u^*$ is some extension of 
  $u^*$ in $\bD\re \ell_*$.
Then
 $U^*$
end-extends   $\max(D)^+$.
 One may take each $u^*$ to be the  $\prec$-leftmost extension of $u$ to be deterministic, but \SDAP\ implies that any extensions will suffice.
(Since $t^*_{\mathbf{i}}$ is either a splitting node or a splitting predecessor,
all possible choices of $u^*$ for $u\in \max(D)^+\setminus B$ are are automatically never  splitting nodes nor splitting predecessors nor coding nodes in $\bD$.)

Let $\{ n_j:j\in \mathbb{N}\}$ be the strictly increasing enumeration of $I$, and 
note  that $n_0> d$.
From now on, Cases  (a) and (b) are  different enough to warrant  separate treatment.
\vskip.1in

\noindent {\bf Case (a).}
If $n_0=d+1$,
then
$D\cup U^*$ is a member of $r_{n_0}[D,\bD]$.
In this case, we let  $U_{n_0}=D\cup U^*$,
and
 let $U_{n_1-1}$ be any member of $r_{n_1-1}[U_{n_0},\bD]$,
 noting that
    $U^*$ is  the only member of $\Ext_{U_{n_0}}(B)$ and that
$h'(U^*)=\varepsilon_*$.
Otherwise,
 $n_0>d+1$.
 In this case,
 take some  $U_{n_0-1}\in r_{n_0-1}[D,\bD]$ such that
 $\max(U_{n_0-1})$ end-extends $U^*$,
 and
 notice that   $\Ext_{U_{n_0-1}}(B)$ is empty.

Now assume that   $j\ge 0$ and 
 we have constructed $U_{n_j-1}\in r_{n_j-1}[D,\bD]$
  so that every member of $\Ext_{U_{n_j-1}}(B)$
 has $h'$-color $\varepsilon_*$.
Fix some  $E\in r_{n_j}[U_{n_j-1} ,\bD]$ and let $Y=\max(E)$.
We will extend the nodes in $Y$  to construct
$U_{n_j}\in r_{n_j}[U_{n_j-1},\bD]$
with the property that all members of $\Ext_{U_{n_j}}(B)$ have the same
 $h$-value $\varepsilon_*$.
This will be achieved  by constructing
 the condition $q\in\bP$, below, and then extending it to some condition $r \le q$ which decides that  all members of
 $\Ext(B)$ coming from the nodes in $\ran(r)$ have $h$-color $\varepsilon_*$.

Let $q(\mathbf{i})$ denote  the  splitting node  in $Y$ and let $\ell_q=\ell_Y$.
For each $i<\mathbf{i}$,
let  $Y_i$ denote  $Y\cap T_i$,
and 
let $J_i\sse K_i$ be a set of  the same cardinality as $Y_i$
and label the members of $Y_i$ as
$\{z_{\al}:\al\in J_i\}$.
Let $\vec{J}$ denote $\prod_{i<\mathbf{i}}J_i$, and
note
that for each $\vec\al\in\vec{J}$,
the set $\{z_{\al_i}:i<\mathbf{i}\}\cup\{q(\mathbf{i})\}$ is a member of $\Ext(B)$.
By   Lemma \ref{lem.compat},
the set $\{p_{\vec\al}:\vec\al\in\vec{J}\}$ is compatible,  and
$p_{\vec{J}}:=\bigcup\{p_{\vec\al}:\vec\al\in\vec{J}\}$ is a condition in $\bP$.

Let
 $\vec{\delta}_q=\bigcup\{\vec{\delta}_{\vec\al}:\vec\al\in \vec{J}\}$.
For $i<\mathbf{i}$ and $\al\in J_i$,
define $q(i,\al)=z_{\al}$.
It follows  that for each
$\vec\al\in \vec{J}$ and $i<\mathbf{i}$,
\begin{equation}
q(i,\al_i)\contains t^*_i=p_{\vec\al}(i,\al_i)=p_{\vec{J}}(i,\al_i),
\end{equation}
and
\begin{equation}
q(\mathbf{i})\contains t^*_{\mathbf{i}}=p_{\vec\al}(\mathbf{i})=p_{\vec{J}}(\mathbf{i}).
\end{equation}
For   $i<\mathbf{i}$ and $\delta\in\vec{\delta}_q\setminus
J_i$,
 let
$q(i,\delta)$ be  any   node in $\bD\re \ell_q$
extending
 $p_{\vec{J}}(i,\delta)$.
Define
\begin{equation}
q=\{q(\mathbf{i})\}\cup \{\lgl (i,\delta),q(i,\delta)\rgl: i<\mathbf{i},\  \delta\in \vec{\delta}_q\}.
\end{equation}
This $q$ is a condition in $\bP$, and $q\le p_{\vec{J}}$.

Take an $r\le q$ in  $\bP$ which  decides some $\ell$ in $\dot{L}_G$ for which   $h'(\dot{b}_{\vec\al}\re \ell)=\varepsilon_*$, for all $\vec\al\in\vec{J}$.
Without loss of generality,
 we may assume that $\ell_r=\ell$.
Since
$r$ forces $\dot{b}_{\vec{\al}}\re \ell=X(r,\vec\al)$
for each $\vec\al\in \vec{J}$,
and since the coloring $h'$ is defined in the ground model,
it follows that
$h'(X(r,\vec\al))=\varepsilon_*$ for each $\vec\al\in \vec{J}$.
Let
\begin{equation}
Y'=\{q(\mathbf{i})\}\cup \{q(i,\al):i<\mathbf{i},\ \al\in J_i\},
\end{equation}
 and let
\begin{equation}
Z'=\{r(\mathbf{i})\}\cup \{r(i,\al):i<\mathbf{i},\ \al\in J_i\}.
\end{equation}

Let $Z$ be the level set
consisting of  the nodes in $Z'$ along with
a  node  $z_y$ in $\bD\re \ell$ extending  $y$, for each
 $y\in
Y\setminus Y'$.
Then $Z$ end-extends $Y$,
and moreover,
letting $U_{n_j}=U_{n_j-1}\cup Z$,
we see that $U_{n_j}$ is a member of $r_{n_j}[U_{n_j-1},\bD]$ such that
$h'$ has value $\varepsilon_*$ on $\Ext_{U_{n_j}}(B)$.
\vskip.1in

\noindent {\bf Case (b).}
Notice that in this case,
 $n_0$ must be at least $d+2$ and that $t^*_{\mathbf{i}}$ is the splitting predecessor of 
 the coding node in $\bD\re \ell_*$, which we shall denote by $c^*$.
Let $U^*$ be as in equation (\ref{eq.U*}).
In Case (b),
all nodes in $U^*$ have length $\ell_*$  except for $t^*_{\mathbf{i}}$, which has length $\ell'_*$.
There is exactly one non-terminal (i.e.\ non-coding) node in $\bD\re \ell_*$ extending $t^*_{\mathbf{i}}$; denote this node by $u^*_{\mathbf{i}}$.
 If $n_0=d+2$,
let $U_{n_0}$ be the tree induced by  $D\cup U^*\cup\{u_{\mathbf{i}}^*,c^*\}$.
Then let $U_{n_1-2}$ be any member of $r_{n_1-2}[U_{n_0},\bD]$.

 If $n_0>d+2$,
the same argument will 
 handle the base case and the induction step.
For the base case,
 let $E$ be a member of $r_{n_0}[D,\bD]$ such that
 $E\re\ell_*$    equals
  $(U^*\setminus\{t^*_{\mathbf{i}}\})\cup\{u^*_{\mathbf{i}}\}$.
(In particular, $\ell_*<\ell_{r_{d+1}(U_{n_0})}$.)
For   $j\ge 1$, supposing
 we have constructed $U_{n_j-2}\in r_{n_j-2}[U_{n_{j-1}},\bD]$
  so that every member of $\Ext_{U_{n_j-2}}(B)$
 has $h'$-color $\varepsilon_*$,
let  $E$ be any member of $r_{n_j}[U_{n_j-2} ,D]$.
In each of these  two cases, 
let $c^E$ denote the coding node in $\max(E)$,
and let  $Y$ denote the set $\max(E)$
but with the two extensions of $\splitpred(c^E)$ in $\max(E)$ deleted and replaced by $\splitpred(c^E)$.

Let 
$\ell_q=\ell_E$, 
and let $q(\mathbf{i})$ denote $\splitpred(c^E)$.
For each $i<\mathbf{i}$,
let  $Y_i$ denote  the set of nodes $y\in Y\cap T_i$
 such that $y$ is a member of  some $X\in \Ext_E(B)$.
Equivalently,
$Y_i$ is the set of those $y\in Y\cap T_i$
such that
$y^+(c^E)\sim (t^*_i)^+(c^*)$.
For each $i<\mathbf{i}$,
take  a set $J_i\sse K_i$ of the same  cardinality as     $Y_i$
and label the members of $Y_i$ as
$\{z_{\al}:\al\in J_i\}$.
Let $\vec{J}$ denote
$\prod_{i<\mathbf{i}}J_i$,
noting that for each $\vec\al\in\vec{J}$,
$\{z_{\al_i}:i<\mathbf{i}\}\cup\{q(\mathbf{i})\}$ is a member of $\Ext(B)$.
By   Lemma \ref{lem.compat},
the set $\{p_{\vec\al}:\vec\al\in\vec{J}\}$ is compatible,  and
$p_{\vec{J}}:=\bigcup\{p_{\vec\al}:\vec\al\in\vec{J}\}$ is a condition in $\bP$.

Let
 $\vec{\delta}_q=\bigcup\{\vec{\delta}_{\vec\al}:\vec\al\in \vec{J}\}$.
For $i<\mathbf{i}$ and $\al\in J_i$,
define $q(i,\al)=z_{\al}$.
It follows  that for each
$\vec\al\in \vec{J}$ and $i<\mathbf{i}$,
\begin{equation}
q(i,\al_i)\contains t^*_i=p_{\vec\al}(i,\al_i)=p_{\vec{J}}(i,\al_i),
\end{equation}
and
\begin{equation}
q(\mathbf{i})\contains t^*_\mathbf{i}=p_{\vec\al}(\mathbf{i})=p_{\vec{J}}(\mathbf{i}).
\end{equation}
For   $i<\mathbf{i}$ and $\delta\in\vec{\delta}_q\setminus
J_i$,
 let $q(i,\delta)$ be  an  extension
 of $p_{\vec{J}}(i,\delta)$ in $T_i$ of length $\ell_q$  satisfying
 \begin{equation}
 q(i,\delta)^+(c_q)
 \sim
 p_{\vec{J}}(i,\delta)^+(c^*),
 \end{equation}
 where $c_q$ denotes the coding node in $\bD\re \ell_q$.
Such nodes $q(i,\delta)$ exist by SDAP.
Define
\begin{equation}
q=\{q(\mathbf{i})\}\cup \{\lgl (i,\delta),q(i,\delta)\rgl: i<\mathbf{i},\  \delta\in \vec{\delta}_q\}.
\end{equation}
This $q$ is a condition in $\bP$, and $q\le p_{\vec{J}}$.

Now take an $r\le q$ in  $\bP$ which  decides some $\ell$ in $\dot{L}_G$ for which   $h'(\dot{b}_{\vec\al}\re \ell)=\varepsilon_*$, for all $\vec\al\in\vec{J}$.
Without loss of generality,
 we may assume that $\ell_r=\ell$.
Since
$r$ forces $\dot{b}_{\vec{\al}}\re \ell=X(r,\vec\al)$
for each $\vec\al\in \vec{J}$,
and since the coloring $h'$ is defined in the ground model,
it follows that
$h'(X(r,\vec\al))=\varepsilon_*$ for each $\vec\al\in \vec{J}$.
Let
\begin{equation}
Z_0=\{r(\mathbf{i})\}\cup \{r(i,\al):i<\mathbf{i},\ \al\in J_i\}.
\end{equation}
Recall  that $\ran(q)\sse Y$, and note that $Z_0$ end-extends $\ran(q)$.

Let  $c_r$ denote the coding  node in $\bD$ of length  $\ell_r$.
 For each
$y\in Y\setminus \ran(q)$,
 choose
a  member  $z_y\supset y$  in $\bD\re \ell_r$
so that
\begin{equation}
z_y^+(c_r)\sim y^+(c_q).
\end{equation}
Again, \SDAP\ ensures the existence of  such $z_y$.
Let
$Z$ be the level set
consisting of
the nodes $z_y$ for
 $y\in Y\setminus \ran(q)$,
 the nodes in $Z_0\setminus \{r(\mathbf{i})\}$,
  and the two nodes in $\bD\re \ell_r$ extending $r(\mathbf{i})$.
Let 
\begin{equation} 
U_{n_j}=
U_{n_j-2}\cup Z\cup (Z\re \ell'_r).
\end{equation}
Then  $U_{n_j}$
is a member of $r_{n_j}[U_{n_j-2},\bD]$.

Now that we have constructed $U_{n_j}$,
let $U_{n_{j+1}-2}$ be any member of
$r_{n_{j+1}-2}[U_{n_j},\bD]$.
This completes the inductive construction.
Let $N=\bigcup_{j<\om}U_{n_j}$.
Then $N$ is a member of $[D,\bD]^*$ and
 for each $X\in\Ext_{N}(B)$,  $h'(X)=\varepsilon_*$.
Thus, $N$ satisfies the theorem.
\end{proof}

\begin{rem}
A simple modification of the proof 
yields the same theorem for 
 structures with \LSDAP$^+$.
(See proof of Theorem 5.4 in   \cite{CDP1}.) 
\end{rem}


\section{Borel sets of  $\mathcal{D}(\bD)$ are completely Ramsey}\label{sec.MainThm}

The main result of this section is  Theorem \ref{thm.best}:
For any enumerated  \Fraisse\  structure $\bK$ satisfying SDAP$^+$,
for each good diagonal coding antichain $\bD$ representing $\bK$,
the  space $\mathcal{D}(\bD)$ of 
 all diagonal antichains $M\sse \bD$ similar to  $\bD$
has the property that all Borel subsets are Ramsey.
The proof generally follows the outline of the Galvin--Prikry Theorem  in \cite{Galvin/Prikry73} with the following  exceptions: 
The proof of  the Nash-Williams-style Theorem \ref{thm.GalvinNW}
uses an asymmetric version of combinatorial forcing as well as  applications of the 
Extended Pigeonhole Principle.
This Principle is also needed
 to show that 
countable unions of completely Ramsey sets are completely Ramsey.
Finally,  working with diagonal coding antichains  requires extra care.

Fix a  \Fraisse\  structure $\bK$ with universe $\mathbb{N}$ satisfying SDAP$^+$.
Fix 
 a good diagonal coding  antichain $\bD$ representing  $\bK$, and let  $\mathcal{D}$ denote $\mathcal{D}(\bD)$.
We hold to the following convention:

\begin{conv}\label{conv.basicmetricopen}
Given  $M\in\mathcal{D}$ and 
  $A\in\mathcal{AD}(M)$,  
when we write 
$[A,M]$,
it is assumed that $\max(A)$ does not contain a splitting  predecessor in $M$.
When  we write 
 $B\in r[A,M]^*$,
 if $\max(B)$ has a splitting node then  it is assumed that that splitting node is not a splitting predecessor in $M$.
\end{conv}

\begin{defn}\label{defn.NWfamily}
A family  $\mathcal{F}\sse\mathcal{AD}$  has the {\em Nash-Williams property} if  for any two distinct members
in $\mathcal{F}$,
neither is an initial segment of the other.
\end{defn}

Nash-Williams families correspond to metrically open sets in $\mathcal{D}$.

\begin{defn}\label{defn.front}
Suppose $M\in\mathcal{D}$ and $B\in\widehat{\mathcal{AD}}(M)$.
A Nash-Williams family $\mathcal{F}\sse r[B,M]^*$ is a {\em front on $[B,M]^*$} if
for each $N\in [B,M]^*$,
there is some $C\in\mathcal{F}$ such that $C\sqsubset N$.
\end{defn}

A front $\mathcal{F}$ on $[B,M]^*$ determines a collection of disjoint (Ellentuck) basic open sets $[C,M]$, $C\in\mathcal{F}$, whose union is exactly $[B,M]^*$.

Recall that for  level sets $X,Y\sse\bD$, we
write $X\sqsubset Y$  exactly when  $X$ and $Y$ have the same cardinality, $\ell_X<\ell_Y$,  and $Y\re \ell_X=X$.
More generally, for $B,F\in\widehat{\mathcal{AD}}$, write 
$B\sqsubseteq F$ 
exactly when  $B= F\rl \ell_B$;
in this case, 
write $B\sqsubset F$ if  also $\ell_B<\ell_F$.
Given $\mathcal{F}\sse \mathcal{AD}$ and
$B\in\widehat{\mathcal{AD}}$, define
\begin{equation}
\mathcal{F}_B=\{F\in\mathcal{F}: B\sqsubset F\}.
\end{equation}
In particular, if $\mathcal{F}\sse r[B,\bD]^*$, then $\mathcal{F}_B=\mathcal{F}$.
If  $\mathcal{F}$ is  a Nash-Williams family, then  $B\in\mathcal{F}$ if and only if $\mathcal{F}_B=\{B\}$.

Given $M\in\mathcal{D}$, let
\begin{equation}
\mathcal{F}|M=\{F\in\mathcal{F}:F\in\mathcal{AD}(M)\}.
\end{equation}
With this notation,  $\mathcal{F}_B|M=\mathcal{F}\cap r[B,M]^*$, for any $B\in\widehat{\mathcal{AD}}$.
For  $F\in\mathcal{AD}$, let
$|F|$ denote the  $k$ for which $F\in\mathcal{AD}_k$.
Given a set $\mathcal{F}\sse \mathcal{AD}$,
let
\begin{equation}
\widetilde{\mathcal{F}}=\{r_k(F):F\in\mathcal{F}\mathrm{\ and\ }k\le |F|\}.
\end{equation}
If $\mathcal{F}$ is a Nash-Williams family, then  $\mathcal{F}$ consists of the $\sqsubseteq$-maximal members of $\widetilde{\mathcal{F}}$.

We now prove an analogue of the Nash-Williams Theorem for our spaces of  good diagonal coding antichains.

\begin{assumption}\label{assumption.important}
Recall that we are under  Convention \ref{conv.basicmetricopen}.
Given 
 $M\in\mathcal{D}$  and 
$A\in \widehat{\mathcal{AD}}(M)$,
 let $d=\depth_M(A)$ and 
$D=r_d(M)$.
Recall that
 $A^+$ denotes the union of $A$  with the set of   immediate extensions in    $\widehat{M}$ of the members of $\max(A)$.
Let $B$ be a member of $\widehat{\mathcal{AD}}(M)$ such that
 $A\sqsubset B\sse A^+$.
In what follows we consider simultaneously the two pairs of cases for triples $(A,B,k)$, from Section \ref{sec.APP}:
\vskip.1in

\begin{enumerate}
\item[]
\begin{enumerate}
\item[\bf Case (a).]
$\max(r_{k+1}(\bD))$ has a splitting node.
\end{enumerate}
\end{enumerate}

\begin{enumerate}
\item[]
\begin{enumerate}
\item[\bf Case (b).]
$\max(r_{k+1}(\bD))$
has a coding node.
\end{enumerate}
\end{enumerate}

\begin{enumerate}
\item[]
\begin{enumerate}
\item[\bf Case (i).]
$k\ge 1$,
$A\in \mathcal{AD}_k$,
and
 $B=A^+$.
 \end{enumerate}
\end{enumerate}

\begin{enumerate}
\item[]
\begin{enumerate}
\item[\bf Case (ii).]
$k\ge 0$, $A$ has at least one node,
$\max(A)\sqsubset \max(B)$, and 
$A = C\rl \ell$ for some $C\in\mathcal{AD}_{k+1}$  and $\ell<\ell_C$ such that 
$r_{k}(C)\sqsubseteq A$ and $B\sqsubseteq C$.
\end{enumerate}
\end{enumerate}
\end{assumption}

\begin{thm}\label{thm.GalvinNW}
Given $M\in\mathcal{D}$, $(A,B,k)$, $d=\depth_M(A)$, and $D=r_d(M)$ as in  Assumption \ref{assumption.important},
let
$\mathcal{F}\sse r[B,M]^*$ be  a Nash-Williams family.
Then   there is an $N\in [D,M]$ such that
either $\mathcal{F}|N$ is a front on $[B,N]^*$
or else
 $\mathcal{F}|N=\emptyset$.
\end{thm}

\begin{proof}
Since $A$ and $B$ are fixed, we shall use lower case $a,b,c,\dots$ to denote members of $\mathcal{AD}$ in this proof.
Recall that  the notation $N\le M$ means that  $N\in [\emptyset, M]$.
We say that 
$N\le M$ {\em accepts} $a\in r[B,N]^*$ if  $\mathcal{F}_a| N$ is a front on $[a,N]$.
We say that 
$N$ {\em widely-rejects} (w-rejects) $a$ if either
\begin{enumerate}
\item[(a)]
$a\not\in r[B, N]^*$; or
\item[(b)]
$a\in  r[B, N]^*$ and $\forall P\in [\depth_N(a),N]\ \exists Q\in[a,P]\ \forall n(r_n(Q)\not\in\mathcal{F})$.
\end{enumerate}
We say that $N$ {\em decides} $a$ if either $N$ accepts $a$ or else $N$ w-rejects $a$.\footnote{After the author had developed this proof, it was pointed out that a similar asymmetric version of combinatorial forcing was developed by Todorcevic in notes for a graduate course in Ramsey theory.  However, those notes do not directly apply to sets of the form $[B,M]^*$, nor do they include the concluding argument in our proof after Lemma \ref{lem.decides}.}
For $n\in\om$,  let  $[n,N]$ denote $[r_n(N),N]$.

\begin{fact}\label{fact.1}
If $N$ accepts $a$, then so does each $P\le N$ with $a\in \mathcal{AD}(P)$.
If $N$ w-rejects $a$, then 
either $a\not\in r[B, N]^*$ and  every $P
\le N$ also rejects  $a$, or else $a\in r[B,N]^*$ and every  $P\in [\depth_N(a),N]$  w-rejects $a$.
\end{fact}

\begin{proof}
Suppose $N$ accepts $a$ and $P\le N$ with $a\in\mathcal{AD}(P)$.
Since $\mathcal{F}_a|N$ is a front on $[a,N]$, it follows that $\mathcal{F}_a|P$ is a front on $[a,P]$.
Hence $P$ accepts $s$.

Suppose $N$ w-rejects $a$.  If $a\not\in \mathcal{AD}(N)$, then
also for each $P\le N$, 
 $a\not\in \mathcal{AD}(P)$  and 
hence
 $P$ w-rejects $a$.
Otherwise,  $a\in \mathcal{AD}(N)$.
Let $n=\depth_N(a)$ and suppose
$P\in [n,N]$. 
Since  $N$ w-rejects $a$,   for each $Q\in [n,N]$   there is an $R\in [n,Q]$ such that for all $m$, $r_m(R)\not\in \mathcal{F}$.
Note that  $P\in [n,N]$ implies $[n,P]\sse[n,N]$; 
so 
 for each $Q\in [n,P]$ there is an $R\in [a,Q]$ such that  for all $m$, $r_m(Q)\not\in \mathcal{F}$. Therefore, $P$ w-rejects $a$.
\end{proof}

\begin{lem}\label{lem.A}
Given $a\in r[B,N]^*$  and $n=\depth_N(a)$,
either $\exists P\in [n,N]$ which  w-rejects $a$, or else $\forall P\in[n,N]\ \exists Q\in[n,P]$ which accepts $a$.
\end{lem}

\begin{proof}
Suppose there is no $P\in[n,N]$ which w-rejects $a$.
Then $\forall P\in[n,N]$, 
\begin{equation}
\exists Q\in [n,P]\ \forall X\in [a,Q]\ \exists m(r_m(X)\in \mathcal{F}).
\end{equation}
Thus,  for all $P\in [n,N]$ there is a $Q\in [n,P]$ such that $\mathcal{F}_a|Q$ is a front on $[a, Q]$;
that is, $Q$  accepts $a$. 
\end{proof}

\begin{fact}\label{fact.2}
\begin{enumerate}
\item[(a)]
For each $a\in r[B,M]^*$, there is an $N\in[\depth_M(a),M]$ which decides $a$.
\item[(b)]
 If $a\in r[B,M]^*$, then $N\in [B,M]^*$ with $a\in\mathcal{AD}(N)$ accepts $a$ if and only if $N$ accepts each $b\in r_{|a|+1}[a,N]$.
\end{enumerate}
\end{fact}

\begin{proof}
For (a),  let  $n=\depth_M(a)$.
By  Lemma \ref{lem.A}, either there is an $N\in [n,M]$ which w-rejects $a$,
or else there is an $N\in [n,M]$ which accepts $a$.

For (b), given the hypotheses,
$N$ accepts $a$ iff $\mathcal{F}_a|N$ is a front on $[a,N]$ iff for each $b\in r_{|a|+1}[a,N]$, 
$\mathcal{F}_a|N$ is a front on $[b,N]$
iff $N$ accepts each  $b\in r_{|a|+1}[a,N]$.
\end{proof}

Recall that $B\in\widehat{\mathcal{AD}}$, but is  not necessarily a member of $\mathcal{AD}$.
We shall say 
 that $N$ {\em accepts} $B$ if $N$ accepts
$a$ for all $a\in r_{k+1}[B,N]^*$.

\begin{fact}\label{fact.claim.1}
If $N\in[B,M]^*$  accepts 
$B$, then $\mathcal{F}_B|N$ is a front on $[B,N]^*$.
\end{fact}

\begin{proof}
For  each 
 $a\in r_{k+1}[B,N]^*$,   $N$ accepts $a$ implies that 
$\mathcal{F}_a|N$ is a front on $[a,N]$.
Since  $[B,N]^*=\bigcup\{[a,N]:a\in  r_{k+1}[B,N]^*\}$, it follows 
 that   $\mathcal{F}|N=\bigcup\{\mathcal{F}_a|N:a\in r_{k+1}[B,N]^*\}$, which is a front on $[B,N]^*$.
\end{proof}

\begin{lem}\label{lem.decides}
There is an $N\in[d,M]$ which 
decides each $a$ in $r[B,N]^*$.
\end{lem}

\begin{proof}
By finitely many applications of 
Fact \ref{fact.2}, 
we obtain an $M_1\in [d+1,M]$ such that $M_1$ decides each 
$a\in r[B,M]^*$  with 
 $a \le_{\mathrm{fin}}r_{d+1}(M)$.
Given $M_i$,
by finitely many applications of 
Fact \ref{fact.2}, 
we obtain an $M_{i+1}\in [d+i+1,M_i]$ such that $M_{i+1}$ decides each 
$a\in r[B,M_i]^*$  with 
 $a \le_{\mathrm{fin}}r_{d+i+1}(M_i)$.
Let $N=\bigcup_{i=1}^{\infty} r_{d+i}(M_i)$, which is the same as  $\bigcup_{i=1}^{\infty} r_{d+i+1}(M_i)$.
Then $N\in [d,M]$
(in fact, $N\in[d+1,M]$)
 and 
for  $a\in r[B,N]^*$,
$M_i$ decides $a$, where $i$ is the index satisfying
 $a\le_{\mathrm{fin}} r_{d+i}(M_i)$.
Since $N\in [d+i,M_i]$, it follows that $N$ decides $a$ in the same way that $M_i$ does.
Thus, $N$ decides all $a\in r[B,N]^*$.
\end{proof}

Now we finish the proof of the theorem. 
Take $N$ as in Lemma \ref{lem.decides} and define a coloring
 $f:r[B,N]^*\ra 2$ by 
$f(a)=0$ if $N$ accepts $a$ and $f(a)=1$ if $N$ w-rejects $a$.
By the Extended Pigeonhole Principle, Theorem \ref{thm.matrixHL}, there is a $P\in [d,N]$ for which $f$ is monochromatic on $r_{k+1}[B,P]^*$.
Now if $f$ has color $0$ on this set, then
$P$ accepts $B$ and 
 by Fact \ref{fact.claim.1},  $\mathcal{F}|P$ is a front on $[B,P]^*$.

Otherwise, $f$ has color $1$ on 
$r_{k+1}[B,P]^*$  so  $P$ w-rejects each member of $r_{k+1}[B,P]^*$. 
Let $P_0=P$.
Apply Theorem \ref{thm.matrixHL} finitely many (possibly $0$) times, to obtain some $P_{1}\in[d+1,P_0]$ such that 
for each $a \in r[B,P_{1}]^*$ with $a
\sse  r_{d+1}(P_{0})$,
all members of 
$r_{|a|+1}[B,P_{1}]^*$ have the same $f$-color.
Since such an $a$ is necessarily in $r_{k+1}[B,P_0]^*$ and  $P_0$ w-rejects $a$, Fact \ref{fact.2} implies that this $f$-color must be $1$.

For 
 $i\ge 1$, 
 we have the following  the induction hypothesis:
 $P_i\in[d+i,P_{i-1}]$  and 
for 
each   $a\in r[B,P_{i-1}]^*$ with $a\sse  r_{d+i}(P_{i-1})$,
$P_i$ w-rejects all members of $r_{|a|+1}[a,P_i]$.
Apply Theorem \ref{thm.matrixHL}
finitely many times to obtain a $P_{i+1}\in [d+i+1,P_i]$ such that $f$ is monochromatic on $r_{|a|+1}[a,P_{i+1}]$ for each $a\in r[B,P_i]^*$ with $a\sse  r_{d+i+1}(P_i)$.
Fix an $a\in r[B,P_i]^*$ with $a\sse r_{d+i+1}(P_i)$.
If  $|a|=k+1$ then $P_{i+1}$ w-rejects
 $a$, since $a\in r_{k+1}[B,P]^*$ and $P_{i+1}\in[B,P]^*$.
Suppose now that $|a|>k+1$.
By   the induction hypothesis, $P_i$ w-rejects $a$ since $a\in r_{|b|+1}[b,P_i]$, where 
$b=r_{|a|-1}(a)\sse r_{d+i}(P_{i-1})$.
Now if the  $f$-color on $r_{|a|+1}[a, P_{i+1}]$ is  $0$, then $P_{i+1}$ accepts $a$ by Fact \ref{fact.2}, a contradiction. 
Hence, $f$ has color $1$ on $r_{|a|+1}[a, P_{i+1}]$; 
in particular, $P_{i+1}$ w-rejects each member of $r_{|a|+1}[a, P_{i+1}]$.

Let $Q=\bigcup_{i=1}^{\infty} r_{d+i}(P_i)$.
Then $Q$ w-rejects each member of $r[B,Q]^*$.
By definition of w-rejects,
for each  $a\in  r[B,Q]^*$,
\begin{equation}\label{eq.dagger}
\forall R\in[\depth_Q(a),Q]\ 
\exists X\in[a,R]\ \forall n(r_n(X)\not\in\mathcal{F})
\end{equation}
Suppose toward a contradiction that there is an $a\in\mathcal{F}|Q$.
Then for all $X\in [a,Q]$, $r_{|a|}(X)=a\in\mathcal{F}$.
So  $Q\in [\depth_{Q}(a),Q]$  and   for all $X\in [a,Q]$, $\exists n(r_n(X)\in\mathcal{F})$.
But this contradicts  (\ref{eq.dagger}).
Thus $\mathcal{F}|Q$ must be empty.
\end{proof}


\begin{defn}\label{defn.CR}
Let  $\mathcal{X}$ be a subset of $\mathcal{D}$.
We say that $\mathcal{X}$ is {\em Ramsey}
if for each $M\in\mathcal{D}$ there is  an $N\le M$ such that either $\mathcal{X}\sse [\emptyset,N]$ or else $\mathcal{X}\cap  [\emptyset,N]=\emptyset$.
 $\mathcal{X}$ is said to be  {\em completely Ramsey (CR)}  if for each $C\in\mathcal{AD}$ and each $M\in \mathcal{D}$,
there is an $N\in [C,M]$ such that either $[C,N]\sse \mathcal{X}$ or else $[C,N]\cap\mathcal{X}=\emptyset$.
We shall say that 
  $\mathcal{X}$  is {\em CR$^*$}
if  given $M\in\mathcal{D}$ and 
$(A,B)$ as in 
Assumption
\ref{assumption.important},
there is an $N\in [D,M]$ such that either $[B,N]^*\sse \mathcal{X}$ or else $[B,N]^*\cap\mathcal{X}=\emptyset$.
\end{defn}

\begin{rem}
Since metrically open sets correspond to  Nash-Williams families,
Theorem \ref{thm.GalvinNW}  implies that 
metrically open sets are not only completely Ramsey  but moreover 
CR$^*$, even when relativized below  some $M\in\mathcal{D}$.
\end{rem}

\begin{lem}\label{lem.complements}
Complements of CR$^*$ sets are CR$^*$.
\end{lem}

\begin{proof}
Suppose $\mathcal{X}\sse \mathcal{D}$ is CR$^*$,
and 
 $M\in\mathcal{D}$, $(A,B,k)$, $d=\depth_M(A)$, and $D=r_d(M)$  are as in  Assumption \ref{assumption.important}.
By definition of CR$^*$, there is an $N\in[D,M]$ such that either $[B,N]^*\sse\mathcal{X}$ or else
 $[B,N]^*\cap\mathcal{X}=\emptyset$.
Letting $\mathcal{Y}=\mathcal{D}\setminus \mathcal{X}$,
we see that either
$[B,N]^*\cap\mathcal{Y}=\emptyset$ or else $[B,N]^*\sse\mathcal{Y}$.
\end{proof}

In the rest of this section,
given $M\in\mathcal{D}$,  endow 
 $[\emptyset,M]$ with the
 subspace topology inherited from  $\mathcal{D}$ with the metric topology.
The next two lemmas  build up   to  Lemma \ref{lem.ctblU}, which will show
that countable unions of CR$^*$ sets are CR$^*$.

\begin{lem}\label{lem.GP8}
Suppose $\mathcal{X}\sse\mathcal{D}$ is CR$^*$.
Then for each $M\in\mathcal{D}$ and each 
$C\in\mathcal{AD}(M)$,
there is an $N\in [C,M]$ such that $\mathcal{X}\cap [\emptyset,N]$ is metrically  open   in $[\emptyset,N]$.
\end{lem}

\begin{proof}
Fix $M\in\mathcal{D}$ and  $C\in\mathcal{AD}(M)$.
Notice that  $[\emptyset,M]=\bigcup_{j<\tilde{j}}[B_j,M]^*$,
where 
$\lgl (A_j,B_j):j<\tilde{j}\rgl$ enumerates  all  pairs $(A,B)$ 
with $\depth_M(A)=\depth_M(C)$
satisfying Assumption
\ref{assumption.important}.

Let $M_{-1}=M$.
 Given $M_{j-1}$ for $j<\tilde{j}$,
$\mathcal{X}$ being CR$^*$ implies 
 there is an  $M_j\in [C,M_{j-1}]$ such that either $[B_j,M_j]^*\sse\mathcal{X}$ or else $\mathcal{X}\cap [B_j,M_j]^*=\emptyset$.
Let $N=M_{\tilde{j}-1}$.
Then $N\in [C,M]$  and
for each $j<\tilde{j}$,
$[B_j,N]^*\sse [B_j,M_j]^*$.
Since 
$[\emptyset,N]=\bigcup_{j<\tilde{j}}[B_j,N]^*$,
it follows that 
\begin{equation}
\mathcal{X}\cap [\emptyset,N]=\bigcup_{j<\tilde{j}} (\mathcal{X}\cap [B_j,N]^*).
\end{equation}
For $j<\tilde{j}$,
if $ [B_j,M_j]^*\sse \mathcal{X}$
then
 $\mathcal{X}\cap [B_j,N]^*= [B_j,N]^*$;
and if $\mathcal{X}\cap  [B_j,M_j]^*=\emptyset$
 then
  $\mathcal{X}\cap [B_j,N]^*=\emptyset$.
Thus,
\begin{equation}
\mathcal{X}\cap[\emptyset,N]=\bigcup_{j\in J}[B_j,N]^*,
\end{equation}
where $J=\{j< \tilde{j}: [B_j,M_j]^*\sse \mathcal{X}\}$.
As each $[B_j,N]^*$ is metrically open in the subspace $[\emptyset,N]$,
$\mathcal{X}\cap[\emptyset,N]$  is  also metrically open in the subspace $[0,N]$.
\end{proof}

\begin{lem}\label{lem.GP9}
Suppose $\mathcal{X}_n$, $n\in\mathbb{N}$, are CR$^*$ sets.
Then for each $M\in\mathcal{D}$ and each $C\in\mathcal{AD}(M)$,
there is an $N\in [C,M]$ such that for each $n\in\mathbb{N}$, $\mathcal{X}_n\cap[\emptyset,N]$ is metrically open in
$[\emptyset,N]$.
\end{lem}

\begin{proof}
Assume the hypotheses and let $d=\depth_M(C)$ and $D=r_d(M)$.
Since  $\mathcal{X}_0$ is CR$^*$,  Lemma \ref{lem.GP8} implies 
there is an $M_0\in [D,M]$ 
and a metrically open set $\mathcal{O}_0\sse\mathcal{D}$ 
satisfying
$\mathcal{X}_0\cap [\emptyset,M_0]=\mathcal{O}_0\cap[\emptyset,M_0]$.
In general, given 
$M_i$,
by  Lemma \ref{lem.GP8} there is some $M_{i+1}\in
[r_{d+i+1}(M_i),M_i]$
and some
 metrically open
 $\mathcal{O}_i\sse\mathcal{D}$
satisfying 
$\mathcal{X}_{i}\cap [\emptyset,M_{i}]=\mathcal{O}_i\cap[\emptyset,M_i]$.
Let $N=\bigcup_{i=0}^{\infty}r_{d+i}(M_i)$.
Then $N$ is a member of $[D,M]$.

Letting $M_{-1}=M$,
note that
 $N\in [r_{d+i}(M_i),M_{i-1}]$ for each $i\in\mathbb{N}$.
It follows that
 for each $i\in\mathbb{N}$,
$\mathcal{X}_{i}\cap [\emptyset,N]=\mathcal{O}_i\cap[\emptyset,N]$.
Hence $\mathcal{X}_{i}\cap [\emptyset,N]$ is metrically open in $[\emptyset,N]$.
\end{proof}

\begin{lem}\label{lem.ctblU}
Countable unions of CR$^*$ sets are CR$^*$.
\end{lem}

\begin{proof}
Suppose $\mathcal{X}_n$, $n\in\mathbb{N}$,  are  CR$^*$ subsets of $\mathcal{D}$, and
let $\mathcal{X}=\bigcup_{n=0}^{\infty}\mathcal{X}_n$.
Let  $(M,A,B,k)$ be  as in Assumption \ref{assumption.important}, and let $d=\depth_M(A)$ and $D=r_d(M)$.
By Lemma \ref{lem.GP9}, there is
a $M'\in [D,M]$ such that for each $n$,
$\mathcal{X}_n\cap[\emptyset,M']$ is metrically open in   $[\emptyset,M']$.
Thus, $\mathcal{X}\cap[\emptyset,M']$ is metrically  open in $[\emptyset,M']$,
 so $\mathcal{X}\cap[\emptyset,M']=\mathcal{O}\cap [\emptyset,M']$ for some metrically open set $\mathcal{O}\sse \mathcal{D}$.

Theorem \ref{thm.GalvinNW}  implies that  $\mathcal{O}$ is
 CR$^*$ in $[\emptyset,M']$.
Hence, 
 there is some $N\in [D,M']$ such that
either $[B,N]^*\sse \mathcal{O}$ or else
$[B,N]^*\cap\mathcal{O}=\emptyset$.
Therefore,
either
\begin{equation}
[B,N]^*  =[B,N]^*\cap[\emptyset,M']\sse \mathcal{O}\cap[\emptyset,M']=\mathcal{X}\cap[\emptyset,M'],
\end{equation}
or else
\begin{align}
[B,N]^*\cap\mathcal{X}&=
[B,N]^*\cap[\emptyset,M']\cap\mathcal{X}\cr
&\sse [B,N]^*\cap[\emptyset,M']\cap\mathcal{O}\cr
&\sse [B,N]^*\cap\mathcal{O}
=\emptyset.
\end{align}
Thus, $\mathcal{X}$ is CR$^*$.
\end{proof}

\begin{thm}\label{thm.best}
Let $\bK$ be an enumerated  \Fraisse\ structure satisfying SDAP$^+$, with finitely many relations of arity at most two.
Let $\bD$ be a good diagonal coding antichain representing $\bK$.
Then the collection of CR$^*$ subsets of $\mathcal{D}(\bD)$ contains all Borel subsets of $\mathcal{D}(\bD)$.
In particular,
Borel subsets of the space $\mathcal{D}(\bD)$  are completely Ramsey, and hence Ramsey.
\end{thm}

\begin{proof}
This follows from  Theorem \ref{thm.GalvinNW} and  Lemmas \ref{lem.complements} and
\ref{lem.ctblU}.
\end{proof}

\begin{rem}
A simple modification of the proof  yields the same result for \LSDAP$^+$ structures. 
\end{rem}


\section{Main Theorems}\label{sec.MainTheorem}

This section contains the main theorem that Borel sets in our spaces of subcopies of a given structure $\bK$ are Ramsey, 
conditions under which analogues of the Ellentuck theorem hold, and 
a Nash-Williams-style corollary recovering exact big Ramsey degrees.


\subsection{Borel sets are Ramsey}\label{subsec.Borel}

We now prove the Main Theorem.
Fix  an enumerated \Fraisse\ structure $\bK$ satisfying SDAP$^+$  and a good diagonal coding antichain $\bD\sse \bU(\bK)$ representing a subcopy of $\bK$.
Recall that the universe of $\bK$ is $\mathbb{N}$.
Each  substructure $\mathbf{M}$ of $\bK$
 is uniquely identified with its universe $\mathrm{M}\sse \mathbb{N}$, which 
in turn,  is uniquely identified with the set of coding nodes $\{c_n:n\in \mathrm{M}\}$.
To avoid any ambiguity,
we will use $T_{\bM}$ (rather than $M$) to denote 
 the subtree  of $\bD$ induced by the set of  coding nodes $\{c_n:n\in \mathrm{M}\}$.
Define 
\begin{equation}
 \mathcal{B}(\bD)=
\{\mathrm{M}\in [\mathbb{N}]^{\mathbb{N}}:
T_{\bM}\in \mathcal{D}(\bD)\}.
\end{equation}
That is, 
$\mathrm{M}\sse\mathbb{N}$ is a member of 
$\mathcal{B}(\bD)$ if and only if 
$\{c_n:n\in \mathrm{M}\}\sse\bD$
and 
the tree induced by $\{c_n:n\in \mathrm{M}\}$
is similar to the tree induced by $\bD$.
Note that $\mathcal{B}(\bD)$
is a subspace of the Baire space.

Let $\bfD$ denote the substructure $\bK\re\bD$,
and let $\lgl d_n:n\in\mathbb{N}\rgl$ be the increasing
enumeration of  the universe $\mathrm{D}$ 
of $\bfD$.
Notice
 that $\lgl c_{d_n}:n\in\mathbf{N}\rgl$
enumerates the coding nodes in $\bD$.
Define 
\begin{equation}
\bK(\bfD)
=\{\bM\le \bfD:
\mathrm{M}
\in  \mathcal{B}(\bD)\}.
\end{equation}
That is, 
$\bK(\bfD)$  is   the  subspace of 
$ {\bK\choose\bK}$ consisting of all substructures $\bM$ of $\bfD$ with  universe $\mathrm{M}\in  \mathcal{B}(\bD)$.
Notice that 
$\bK(\bfD)$ is identified with a  subspace  of 
the Baire space via its identification with $\mathcal{B}(\bD)$.
For $\bM\in \bK(\bfD)$,  we will let $\bK(\bM)$ denote the {\em cube} of all substructures of $\bM$ in $\bK(\bfD)$.

For  $\bM\in \bK(\bfD)$, 
let $\lgl m_i:i\in\mathbb{N}\rgl$ be the increasing enumeration of $\mathrm{M}$.
Then  
increasing bijection 
 $m_i\mapsto d_i$ 
induces 
an isomorphism from $\bM$ to $\bfD$,
and $c_{m_i}\mapsto c_{d_i}$ induces a similarity map from 
 $T_{\bM}$ to  $\bD$.
Given  $n\in \mathbb{N}$,
define $\bM_n=\bM\re \{m_i:i<n\}$.
Let
\begin{equation}\label{eq.r_nRado}
\AKD
=\{\bM_n:\bM\in\bK(\bfD)\mathrm{\ and\ } n\in\mathbb{N}\}.
\end{equation}
For $\bfA\in \AKD$ and $\bM\in\bK(\bfD)$,
write $\bfA \sqsubset \bM$ if and only if $\bfA=\bM_n$ for some $n$.
Define
\begin{equation}
[\bfA,\bM]=\{\bN\in\bK(\bfD): \bfA\sqsubset \bN\}.
\end{equation}
These are the basic open sets for the Ellentuck topology on $\bK(\bfD)$ 
 corresponding to the basic Ellentuck open sets $[\mathrm{A},\mathrm{M}]$ in the Baire space, where $\mathrm{A}$ and $\mathrm{M}$ are the universes of $\bfA$ and $\bM$, respectively.
The basic open sets for the metric topology on 
$\bK(\bfD)$  are those of the form $[\bfA,\bfD]$, where $\bfA\in \AKD$.

Let $\theta:\bK(\bfD)\ra \mathcal{D}(\bD)$ denote the map  which sends each
$\bM\in \bK(\bfD)$
to the tree $T_{\bM}$ in 
$\mathcal{D}(\bD)$.
This map is certainly a bijection. 
We will  show that $\theta$ is in fact a homeomorphism
between these two spaces with their metric topologies.

For $n\in\mathbb{N}$,
let $k_n$ denote the least integer such that $c^{\bD}_{n-1}\in r_{k_n}(\bD)$.
Since each $T\in \mathcal{D}(\bD)$ is similar to $\bD$,
it follows that $k_n$ is the least integer such that the $(n-1)$-st coding node of $T$ is in $r_{k_n}(T)$.
In particular,  $k_n$ is least such that 
$\bfD_n=\bK\re r_{k_n}(\bD)$.
For the following lemma, 
recall that since $\bD$ is a {\em good} diagonal coding antichain,
there is some $n_{\bD}$ such that for each $n\ge n_{\bD}$,
there is a one-to-one correspondence between the nodes in  
$\max(r_{k_n}(\bD))^+$
and the
 $1$-types over $\bfD_n$.

\begin{lem}\label{lem.thetacts}
Suppose  $\bM\in\bK(\bfD)$ and 
$\bfA=\bM_n$, where $n\ge n_{\bD}$. 
Then $\theta([\bfA,\bfM])=[r_{k_n}(T_{\bM}),T_{\bM}]$.
\end{lem}

\begin{proof}
Since $n\ge n_{\bD}$,
there is a one-to-one correspondence between the nodes in $\max(r_{k_n}(T_{\bM}))^+$ and the $1$-types over $\bfA$.
For 
 $\bN\in [\bfA,\bM]$,
$\bN$ extends $\bfA$  to some isomorphic subcopy  of $\bM$, and 
 $T_{\bN}$ is a subtree of  $T_{\bM}$.
In order for $\bN$ to be isomorphic to $\bM$,
each $1$-type over $\bfA$ must be represented by a node in
$\max(r_{k_n}(T_{\bN}))^+$.
The only way this is possible is if 
 $r_{k_n}(T_{\bN})
=r_{k_n}(T_{\bM})$.
Thus, letting $A=r_{k_n}(T_{\bM})$,
\begin{align}
\theta([\bfA,\bfM])&=
\{T_{\bN}:\bN\in[\bfA,\bM]\}\cr
&=\{T_{\bN}:\bN\in[\bfA,\bM]\mathrm{\ and\ }r_{k_n}(T_{\bN})=r_{k_n}(T_{\bM})\}\cr
&=\{T_{\bN}:A\sqsubset T_{\bN}\mathrm{\ and\ }T_{\bN}\le T_{\bM}\}\cr
&=[A,T_{\bM}].
\end{align}
\end{proof}

Thus, $\theta$ takes the   basic Ellentuck open set $[\bM_n,\bM]$ to the basic Ellentuck  open set $[r_{k_n}(T_{\bM}),T_{\bM}]$ whenever $n\ge n_{\bD}$.
Furthermore, 
$\theta$ is a homeomorphism from $\bK(\bfD)$ with its metric topology  to $\mathcal{D}(\bD)$ with its metric topology, as follows from the next lemma.

\begin{lem}\label{lem.thetametricopenpres}
The map $\theta$ takes each basic metrically open set  in $\bK(\bfD)$ to a metrically open set in $\mathcal{B}(\bD)$, and $\theta^{-1}$ takes each basic metrically open set in 
$\mathcal{B}(\bD)$
to a metrically open set in  $\bK(\bfD)$.
\end{lem}

\begin{proof}
Let $[\bfA,\bfD]$ be a basic open set in the metric topology on $\bK(\bfD)$, and let $n$ be the number of vertices in $\bfA$.
Then 
\begin{align}
\theta([\bfA,\bfD])
&= \bigcup\{[r_{k_n}(T_{\bM}),\bD]:\bfA\sqsubset \bM\}\cr
&= \bigcup\{[B,\bD]:B\in\mathcal{AD}_{k_n}\mathrm{\ and \ } \bfD\re B=\bfA\}
\end{align}
which is a countable union of metrically open sets in $\mathcal{D}(\bD)$.
Conversely, given a basic open set $[A,\bD]$ in the metric topology on $\mathcal{D}(\bD)$,
we may without loss of generality assume that $A\in\mathcal{AD}_{k_n}$ for some $n$.
Let $n'$ denote the least integer such that 
 for each $\bM\in\bK(\bfD)$,
\begin{equation}
r_{k_n}(T_{\bM_{n'}})
=r_{k_n}(T_{\bM}).
\end{equation}
Then 
\begin{align}
\theta^{-1}([A,\bD])
&= \{\bM\in\bK(\bfD): T_{\bM}\in [A,\bD]\}\cr
&=\bigcup\{[\bfB,\bfD]:\bfB\in\AKD_{n'}\mathrm{\ and\ }r_{k_n}(T_{\bfB})=A\}, 
\end{align}
which is a countable union of basic metrically open sets in $\bK(\bfD)$.
\end{proof}

A set $\mathcal{X}\sse\bK(\bfD)$ is {\em  Ramsey}
if for any  $\bM\in\bK(\bfD)$,
there is some $\bN\le \bM$ in $\bK(\bfD)$ such that either
$\bK(\bN)\sse \mathcal{X}$
 or else
$\bK(\bN)\cap \mathcal{X}=\emptyset$.
A  set $\mathcal{X}\sse\bK(\bfD)$ is {\em completely Ramsey}
if for any $\bfA\in\AKD$ and $\bM\in\bK(\bfD)$,
there is some $\bN\in [\bfA,\bM]$ such that either
$[\bfA,\bN]\sse \mathcal{X}$
 or else
$[\bfA,\bN]\cap \mathcal{X}=\emptyset$.

\begin{thm}\label{thm.main}
Let $\bK$ be an enumerated \Fraisse\ structure satisfying {\SDAP$^+$} (or {\LSDAP$^+$}) with finitely many relations of arity at most two,  let   $\bD$ be a good diagonal coding antichain, and let $\bfD=\bK\re\bD$.
Then every Borel subset  $\mathcal{X}\sse \bK(\bfD)$ is completely 
Ramsey, and hence Ramsey.
\end{thm}

\begin{proof}
Let $\mathcal{X}$ be a Borel subset of
$\bK(\bfD)$, and suppose $\bfA\in\AKD$ and $\bM\in \bK(\bfD)$.
If $[\bfA,\bM]=\emptyset$ then
we are done, 
so assume that 
$[\bfA,\bM]$ is non-empty.
By shrinking $\bM$ if necessary, we may assume that 
 $\bfA$ is an initial segment of $\bM$.
Let $n$ be the integer such that $\bfA=\bM_n$.
By Lemma \ref{lem.thetacts},
$\theta([\bfA,\bM])=[r_{k_n}(T_{\bM}),T_{\bM}]$.
Let $A$ denote $r_{k_n}(T_{\bM})$.

Let $\mathcal{Y}$ be the $\theta$-image of 
$\mathcal{X}$, noting that
$\mathcal{Y}$ is Borel in $\mathcal{D}(\bD)$ with the metric topology
 by Lemma \ref{lem.thetametricopenpres}.
Apply 
 Theorem \ref{thm.best} 
to obtain an 
$N\in [A,T_{\bM}]$ such that either $[A,N]\sse\mathcal{Y}$ or else
$[A,N]\cap\mathcal{Y}=\emptyset$.
Let $\bN=\bfD\re N$.
Then $T_{\bN}=N$, $A=r_{k_n}(T_{\bN})$,
and $\theta^{-1}([A,N])
=\theta^{-1}([r_{k_n}(T_{\bN}),T_{\bN}])
=[\bfA,\bN]$, by  Lemma \ref{lem.thetacts}.
Thus, either $[\bfA,\bN]\sse\mathcal{X}$
or else 
$[\bfA,\bN]\cap\mathcal{X}=\emptyset$.

Minor modifications of the proofs yield the same result for structures with \LSDAP$^+$.
\end{proof}



\subsection{Topological Ramsey spaces of homogeneous structures}\label{subsec.trs}

\begin{thm}\label{thm.SDAPEllentuck}
Let $\bK$ be  any one of the following structures 
with universe $\mathbb{N}$:
The rationals, $\bQ_n$, $\bQ_{\bQ}$, and or any \Fraisse\  structure
satisfying \SDAP$^+$  or \LSDAP$^+$ for which the coding tree of $1$-types  $\bU(\bK)$ 
has the property that on any given level of 
  $\bU(\bK)$, only the coding node splits.
Then the spaces $\mathcal{D}(\bD)$, where $\bD$ is a diagonal coding antichain for $\bK$, are actually topological Ramsey spaces.
\end{thm}

\begin{proof}
For structures as in the theorem statement, 
 it is straightforward to check that  Todorcevic's Axiom \bf A.3\rm(2)  holds.
(This is the axiom which fails for the Rado graph and  similar structures if one works with  good diagonal antichains.)
It is simple to check that Axioms \bf A.1\rm, \bf A.2\rm,  and  \bf A.3\rm(1) hold, and Axiom \bf A.4 \rm is a special case of Theorem \ref{thm.matrixHL} (these  in fact hold for all structures considered in this paper).
Then by Todorcevic's Abstract Ellentuck Theorem in \cite{TodorcevicBK10},
the spaces $\mathcal{D}(\bD)$  satisfy analogues of Ellentuck's Theorem.
\end{proof}


\subsection{Exact big Ramsey degrees from  infinite-dimensional Ramsey theory}\label{subsec.envelopes}

Let $\bD$ be a good diagonal coding  antichain for $\bK$, and let $M\in\mathcal{D}(\bD)$.
Given a finite antichain of coding nodes
$A\sse M$,
let $\lgl c^A_j:j<n\rgl$ enumerate the coding nodes in $A$
and let $\bfA$ denote the structure $\bK\re A$.
Recall  that we identify $A$ with the tree which it induces, and that $\bfA_j$ denotes $\bfA\re \{c^A_{i}:i<j\}$.
Let $k$ be least such that $A\sse r_k(M)$.
An {\em envelope} 
 $E(A)$ of $A$  in $M$
 is  a minimal set of nodes in $r_{k+1}(M)$
containing $A$ such that 
for each $j<n$, the splitting predecessor of $c^A_j$ in $M$ is in $E(A)$, and 
each  $1$-type over $\bfA$ is represented by exactly  one  maximal node in $E(A)$.

Envelopes can be made canonically as follows:
 First, add  all the splitting predecessors of coding nodes in $A$ and 
extend them $\prec$-leftmost in $M$ to length  $\ell^A_{n-1}+1$;
let $A'$ denote 
 this extension of $A$.
Then proceed by induction on $j<n$:
For each $1$-type $\tau$ over $\bfA_1$ not already represented by a node in  $A'$,
add one node $t$ in $M$ of length $\ell^A_0+1$
such that  $t/\bfA_0\sim \tau$;
let 
 $E_0$ denote the set of these nodes of length  $\ell^A_0+1$.
Whenever  there is a choice of  more than one node $t$, add the $\prec$-leftmost such node.
Given $E_{j-1}$ for   $1\le j<n$,
 for each $1$-type $\tau$ over 
 $\bfA_{j+1}$  which is not represented by any node in 
$A'\re (\ell^A_{j}+1)$,
take the $\prec$-leftmost node $s$ in 
$E_{j-1}\cup A'\re (\ell^A_{j-1}+1)$
such that $s/\bfA_j\sim \tau/\bfA_j$,
and extend $s$ $\prec$-leftmost to a node $t$ in $M$ of length 
$\ell^A_j+1$ such that $t/\bfA_{j+1}\sim\tau$.
Let $E_j$ denote the set of these nodes of
length $\ell^A_j+1$.
Then, let $E(A)=A'\cup \bigcup_{j<n}E_j$.

Notice that for each $M\in\mathcal{D}(\bD)$,
every finite antichain $A$ of coding nodes in $M$ has such an envelope in $M$.
Moreover, 
for any $A,B\sse M$ such that $A\sim B$,
the canonical construction of envelopes produces envelopes $E(A)$ and $E(B)$ such that $E(A)\sim E(B)$.






Now, given a good diagonal coding antichain $\bD$ and a finite antichain $A\sse\bD$ with $n$ coding nodes,
let $E(A)$ be the canonical envelope of $A$ in $\bD$.
Define 
$\bE$ to be a good diagonal coding antichain contained in 
$\bD$ such that 
$\bE\re (\ell^A_{n-1}+1)= E(A)$,
and above $E(A)$, each $1$-type over an initial structure of $\bE$ is represented by exactly one node in $\bE$.

The following theorem of Coulson--Dobrinen--Patel in \cite{CDP2} is  recovered  as a Nash-Williams style corollary 
from 
 the Main Theorem in this paper.

\begin{cor}\label{cor.NW}
Let $\bK$ be an enumerated \Fraisse\ structure satisfying {\SDAP$^+$} (or {\LSDAP$^+$}) with finitely many relations of arity at most two, and let $\bD$ be a good diagonal coding antichain representing a copy of $\bK$.
Let $A\sse \bD$ be a finite diagonal antichain, and let $f$ color all similarity copies of $A$ in $\bD$ into finitely many colors.
Then there is a good diagonal coding antichain $\bE\sse\bD$ representing $\bK$ in which all copes of $A$ have the same color.
\end{cor}

\begin{proof}
Let $\bE$ be an end-extension of the envelope $E(A)$ in $\bD$ to a good diagonal coding antichain,
and let $f$ color all similarity copies of $A$ in $\bE$ into finitely many colors. 
Let  $k$ be the least  integer  such that $r_k(\bE)$ contains $A$.
Notice that  for each $M\in  \mathcal{D}(\bE)$,
$r_k(M)\sim E(A)$,
 so 
the coding nodes in 
any $C\in\mathcal{AD}_k(\bE)$
 induce a tree similar to $A$;
denote this tree by $C_A$.
Moreover, for each similarity copy $B$ of $A$ in $\bfE$,
the canonical envelope $E(B)$ in $\bE$ is in $\mathcal{AD}_k(\bE)$.
Thus,
$f$ induces a  coloring $g$ on $\mathcal{AD}_k(\bE)$ by $g(C)=f(C_A)$.
This in turn 
induces an open, hence Borel, coloring $h$ on $\mathcal{D}(\bE)$
via
$h(M)=g(r_k(M))$.
By Theorem \ref{thm.main}, there is an $N\in \mathcal{D}(\bE)$ on which $h$ is constant.  
Thus, $f$ is constant on the similarity copies of $A$ in $N$.
\end{proof}

\begin{rem}
As pointed out in the introduction, the fact that the number of similarity types of  diagonal antichains 
yields the {\em exact} big Ramsey degrees 
is a theorem of Coulson--Dobrinen-Patel in \cite{CDP2}.
\end{rem}






\bibliographystyle{amsplain}
\bibliography{references}

\end{document}